\documentclass[11pt]{article}

\usepackage{verbatim}
\usepackage{float}

\makeatletter

\usepackage[dvipsnames]{xcolor}

  \usepackage{makecell}

\usepackage[T1]{fontenc}

\usepackage[margin=3cm]{geometry}
\usepackage{amsfonts,amssymb,stmaryrd,amsthm,amsmath}

\usepackage{graphicx,tikz,tikz-cd, tabularx}
\usepackage{quiver}
\usepackage{mathtools}

\usepackage{hyperref}

\usepackage[utf8]{inputenc}

\newtheorem{prop}{Proposition}[section]
\newtheorem{theo}[prop]{Theorem}
\newtheorem{lemma}[prop]{Lemma}
\newtheorem{coro}[prop]{Corollary}

\theoremstyle{definition}
\newtheorem{definition}[prop]{Definition}
\newtheorem{remark}[prop]{Remark}
\newtheorem{example}[prop]{Example}

\usepackage{soul}

\makeatletter
\renewcommand\@biblabel[1]{#1}
\makeatother
\newcommand\DynkinNodeSize{2mm}
\newcommand\DynkinArrowLength{3mm}
\tikzset{
  dnode/.style={
    circle,
    inner sep=0pt,
    minimum size=\DynkinNodeSize,
    fill=white,
    draw},
  middlearrow/.style={
    decoration={markings,
      mark=at position 0.6 with
      {\draw (0:0mm) -- +(+135:\DynkinArrowLength); \draw (0:0mm) -- +(-135:\DynkinArrowLength);},
    },
    postaction={decorate}
  },
  leftrightarrow/.style={
    decoration={markings,
      mark=at position 0.999 with
      {
      \draw (0:0mm) -- +(+135:\DynkinArrowLength); \draw (0:0mm) -- +(-135:\DynkinArrowLength);
      },
      mark=at position 0.001 with
      {
      \draw (0:0mm) -- +(+45:\DynkinArrowLength); \draw (0:0mm) -- +(-45:\DynkinArrowLength);
      },
    },
    postaction={decorate}
  },
  sedge/.style={
  },
  dedge/.style={
    middlearrow,
    double distance=0.5mm,
  },
  tedge/.style={
    middlearrow,
    double distance=1.0mm+\pgflinewidth,
    postaction={draw}, 
  },
  infedge/.style={
    leftrightarrow,
    double distance=0.5mm,
  }
}

\newcommand\revddots{\mathinner{\mkern1mu\raise\p@\vbox{\kern7\p@\hbox{.}}\mkern2mu\raise4\p@\hbox{.}\mkern2mu\raise7\p@\hbox{.}\mkern1mu}}
\makeatletter
\def\revddots{\mathinner{\mkern1mu\raise\p@\vbox{\kern7\p@\hbox{.}}\mkern2mu\raise4\p@\hbox{.}\mkern2mu\raise7\p@\hbox{.}\mkern1mu}}
\makeatother

\newcommand\bmat{\begin{pmatrix}}
\newcommand\emat{\end{pmatrix}}
\newcommand\longto{{\, \longrightarrow \, }}

\newcommand\RR{\mathbb{R}}\newcommand\PP{\mathbb P}
\newcommand\CC{\mathbb{C}}
\newcommand\NN{\mathbb{N}}\newcommand\ZZ{\mathbb{Z}}
 \newcommand\KK{\mathbb{K}}\newcommand\AAA{\mathbb A}
\newcommand\GG{\mathbb{G}}

\newcommand\inv{{^{-1}}}
\newcommand\Li{{\mathcal
    L}}
    
    \newcommand\Fc{{\mathcal F}}
    \newcommand\Rc{{\mathcal R}}

\newcommand\Spec{{\operatorname{Spec}}}
\newcommand\Specbf{{\operatorname{\textbf{Spec}}}}
\newcommand\Proj{{\operatorname{Proj}}}

\newcommand\Pic{{\operatorname{Pic}}}
\newcommand\divv{{\operatorname{div}}}
\newcommand\Divv{{\operatorname{Div}}}
\newcommand\Sym{{\operatorname{Sym}}}

\newcommand\rddots
   {\mathinner{\mkern1mu
       \raise 1pt\hbox{.}\mkern2mu
       \raise 4pt\hbox{.}\mkern2mu
       \raise 7pt\hbox{.}\mkern1mu}}

\newcommand\Orb{{\mathcal O}}
\newcommand\Hom{{\operatorname{Hom}}}

\newcommand\Id{{\operatorname{Id}}}

\newcommand\HH{{\operatorname{H}}}

\newcommand{\clu}{\mathcal{A}}
\newcommand{\uclu}{\overline{\mathcal{A}}}
\newcommand{\upp}{\mathcal{U}}

\newcommand{\val}{\mathcal{V}}
\newcommand{\cval}{\mathcal{C}\mathcal{V}}

\newcommand\coef{{\operatorname{Cf}}}
\newcommand{\lX}{{\mathfrak{X}}}

\newcommand{\lif}{\upharpoonleft \kern-0.35em}
\newcommand{\spc}{\kern0.2em}

\newcommand{\ii}{\mathbf{i}}

\newcommand\Cox{\operatorname{Cox}}

\title{Cluster structures on Cox rings}
\author{Luca Francone}
\date{}

\begin{document}

\maketitle

\begin{abstract}

We use the results of \cite{francone2023minimal} to construct a graded cluster algebra structure on the Cox ring of a smooth complex variety $Z$, depending on a base cluster structure on the ring of regular functions of an open subset $Y$ of $Z$. 
After considering some elementary examples of our technique, including toric varieties, we discuss the two main applications.
First: if $Z$ is a flag variety and $Y$ is the open Schubert cell, we prove that our results recover, by geometric methods, a well known construction of Geiss, Leclerc and Schr\"oer \cite{geiss2008partial}.
    Second: we define the \textit{diagonal partial compactification} of a (finite type) cluster variety, and prove that its Cox ring is a  graded upper cluster algebra.
 Along the way, we explain how similar constructions can be done if we replace the Cox ring with a ring of global sections of a sheaf of divisorial algebras on $Z$.

\end{abstract}

\tableofcontents

\section{Introduction}

Let $Z$ be a  smooth complex algebraic variety whose Picard group $\Pic(Z)$ is free and finitely generated.
The Cox ring $\Cox(Z)$ of $Z$ is a $\Pic(Z)$-graded algebra, introduced by Hu and Keel \cite{hu2000mori}, generalizing a well known construction of Cox \cite{cox1995homogeneous}.
The homogeneous  components of $\Cox(Z)$ can be identified with global sections of line bundles on $Z$.
Namely, we have a decomposition and isomorphisms

$$
\Cox(Z)= \bigoplus_{[\Li] \in \Pic(Z) } \Cox(Z)_{[\Li]}, \qquad  \Cox(Z)_{[\Li]} \simeq \HH^0(Z, \Li).
$$
The multiplication of homogeneous elements of the ring $\Cox(Z)$ is, roughly speaking, given by the tensor product of the corresponding sections. 
Cox rings have proven to be fundamental objects in birational geometry.
For instance, if $\Cox(Z)$ is finitely generated as a $\CC$-algebra, all the small modifications of $Z$ can be obtained as GIT quotients of a natural torus action on the spectrum of the ring $\Cox(Z).$
Moreover, Cox rings played a central role in the groundbreaking work of Birkar, Cascini, Hacon, and McKernan \cite{Birkar2010}, which marked a significant advancement in the development of the minimal model program.

\bigskip

Cluster and upper cluster algebras are  commutative rings of combinatorial nature, defined by Fomin and Zelevinsky \cite{fomin2002clusterfoundation} \cite{berenstein2005cluster3}, and that in the last twenty years have had a tremendous impact on many areas of mathematics.
We refer to \cite{keller2012cluster} for an excellent survey on cluster algebras and their connection to other subjects.
We prefer to postpone the precise definition
 of cluster and upper cluster algebras in Section \ref{Sec: cluster algebras reminders}, to avoid unnecessary technicalities in this introduction.
 
\bigskip

The main contribution of this paper is the following. 
Let $Y$ be an open subset of $Z$. 
Suppose that the ring of regular functions $\Orb_Y(Y)$ on $Y$ is an upper cluster algebra $\uclu$, satisfying the technical but mild hypothesis described in Theorem \ref{thm: min mon lifting}.
Then, the following result holds.

\begin{theo}
\label{thm: intro 1}
  Assume that $\Orb_Y(Y)^\times = \CC^\times$ and $\Pic(Y)= \{0\}$. Then, the Cox ring of $Z$ contains a $\Pic(Z)$-graded sub-algebra $\lif \uclu$ which is an upper cluster algebra. 
Moreover, there exists an homogeneous element $M$ in $\lif \uclu$, such that the equality between the localized rings $\lif \uclu_M$ and $ \Cox(Z)_M$ holds. 
\end{theo}

In the notation of Theorem \ref{thm: intro 1}, the upper cluster algebra $\lif \uclu$ is obtained by means of a geometric and combinatorial procedure developed in \cite{francone2023minimal}, called \textit{minimal monomial lifting}.
The minimal monomial lifting can be viewed as an homogenization technique that uses  geometric information about $Z$ to enhance  the algebra $\uclu$ into the graded algebra $\lif \uclu$, while retaining the essential combinatorial properties of the upper cluster algebra $\uclu$.

In general, the inclusion of the upper cluster algebra $\lif  \uclu$ into the ring $\Cox(Z)$ may be strict.
Although this phenomenon, the importance of the upper cluster algebra $\lif \uclu$ lies in the result below. 

\begin{theo}
\label{thm: intro 2}
    There exists at most one $\Pic(Z)$-graded upper cluster algebra $ \uclu^{\dagger}$ satisfying the compatibility condition with $\uclu$ described in Theorem \ref{thm: uncity min mon lifting}, and such that $\Cox(Z)=  \uclu^{\dagger}$ as graded algebras. 
    If $ \uclu^{\dagger}$ exists, then we have that $ \uclu^{\dagger}= \lif \uclu.$
\end{theo}

Essentially, $\lif \uclu$ provides the only possible way to identify the graded ring $\Cox(Z)$ with a graded upper cluster algebra compatible with $\uclu$.
Theorem \ref{thm: intro 2} is a less precise formulation of the more detailed and technical Theorem \ref{thm: uncity min mon lifting}, which is a general result on minimal monomial lifting.
The question  of whether $\lif \uclu$ coincides with $\Cox(Z)$ can be studied, geometrically, by means of Proposition \ref{prop: equality conditions}.

\bigskip

In this paper, after discussing some elementary examples of cluster structures on Cox rings constructed by means of Theorem \ref{thm: intro 1}, including toric varieties, we give two main applications.
 
\bigskip

\textbf{1. Flag varieties and a construction of Geiss, Leclerc and Schr\"oer.} 
Let  $Z^-$ be a partial flag variety of a semisimple simply connected algebraic group $G$. 
The Cox ring of $Z^-$ is an important object for the representation theory of $G$.
Indeed, by the the Borel-Weil theorem, its homogeneous components can be naturally identified with the irreducible representations of $G$.
By the work of Goodearl and Yakimov \cite{goodearl2021integral} and of Geiss, Leclerc and Schr\"oer \cite{geiss2011kac}, the ring of regular functions of the open Schubert cell $Y^-$ can be identified with an upper cluster algebra.
Theorem \ref{thm: intro 1} implies that the ring $\Cox(Z^-)$ contains a graded upper cluster sub-algebra obtained by minimal monomial lifting of this base cluster structure on $\Orb_{Y^-}(Y^-).$
We establish the following result, which is a loose version of Theorem \ref{thm: complete flag}, and has essentially been proved in \cite{francone2023minimal}.
\begin{theo}
\label{thm: intro 3}
    Assume that $Z^-$ is the complete flag variety of $G$. That is, $Z^-=B^- \setminus G$ for some Borel subgroup $B^-$ of $G$. 
    Let $\uclu$ be the upper cluster algebra structure on $\Orb_{Y^-}(Y^-)$ considered in \cite{goodearl2021integral} and \cite{geiss2011kac}.
    Then, the upper cluster algebra $\lif \uclu$ of Theorem \ref{thm: intro 1} coincides with the ring $\Cox(Z^-).$
\end{theo}

When $Z^-$ is not the complete flag variety, the picture is much more complicated, and is not studied here.

\bigskip

Assume from now on that the group $G$ is simple and simply laced.
The ring of regular functions on  $Y^-$ admits a second cluster structure considered in \cite{geiss2008partial}. 
This cluster structure is related to the one described in \cite{goodearl2021integral} and \cite{geiss2011kac}, but different. 
In the paper \cite{geiss2008partial}, the authors enhance this cluster structure to produce a cluster algebra inside the multi-homogeneous coordinate ring of $Z^-$.
This ring is canonically identified with $\Cox(Z^-)$.
Theorem \ref{thm:Cox GP geiss e minimal mon lifting} states that the construction of the present paper recovers, in this special setting, the cluster structure obtained by Geiss, Leclerc and Schr\"oer.
Theorem \ref{thm:Cox GP geiss e minimal mon lifting} is a remarkable result, primarily because the methods of this paper are significantly different than the ones of \cite{geiss2008partial}. 
The main difference is that Geiss-Leclerc-Schr\"oer's technique for enhancing the cluster structure on $\Orb_{Y^-}(Y^-)$ to the ring $\Cox(Z^-)$ relies on representation theoretic properties of the preprojective algebra of the Dynkin quiver of $G$, which are closely related to the chosen cluster structure on $\Orb_{Y^-}(Y^-)$.
On the other hand, our construction only relies on how the base cluster structure on $\Orb_{Y^-}(Y^-)$ interplays with the geometry of $Z^-$. 
Geiss-Leclerc-Schr\"oer's construction has the advantage of being highly specific, allowing for fine control over the resulting cluster structure on $\Cox(Z^-)$,
 thanks to its  representation theoretic interpretation.
 On the other hand, the construction presented here has the advantage of being more elementary and much less specific, thus allowing for a broader range of applications. 

 \bigskip
 
Note that Kadhem \cite{kadhem2023cluster} generalized part of Geiss-Leclerc-Schr\"oer's construction to partial flag varieties of arbitrary semisimple simply connected algebraic groups $G$.
The same proof of Theorem \ref{thm:Cox GP geiss e minimal mon lifting} can be used to interpret Kadhem's construction as a special case of the one described in this paper.
 
\bigskip

\textbf{2. The diagonal partial compactification of the finite cluster variety.}
Let $\uclu$ be an upper cluster algebra. 
By a construction of Fock and Goncharov \cite{fock2009cluster}, the algebra $\uclu$ can be interpreted as the ring of regular functions on a scheme $\mathcal{Y}$ called \textit{cluster variety}.
 The scheme $\mathcal{Y}$ is essentially obtained by glueing infinitely many  isomorphic varieties called \textit{cluster tori}.
The \textit{finite cluster variety} is an open subset $Y$ of $\mathcal{Y}$  covered by an appropriate finite set of cluster tori.
We refer to Section \ref{sec: diag comp cluster manifold} for a precise definition of the variety $Y.$
Under mild hypothesis, we have that $\Orb_Y(Y)= \uclu$.
In Section \ref{sec: diag comp cluster manifold}, we define a partial compactfication $Z$ of $Y$, to which we refer as the \textit{diagonal partial compactification} of $Y$.
If the upper cluster algebra $\uclu$ is factorial, we show that we can apply  Theorem \ref{thm: intro 1} to the pair $(Z,Y)$, and we prove the following result. (See Theorem \ref{thm:principal clust variety} for a more precise statement).

\begin{theo}
    \label{intro:Cox 3}
   We have an equality of $\Pic(Z)$-graded algebras $\Cox(Z)=\lif \uclu$.
\end{theo}

The construction of the scheme $Z$ is, to my best knowledge, new. 
In would be interesting to understand if the scheme $Z$ can be extended to a global partial compactification of the whole cluster variety $\mathcal{Y}$.

\bigskip
 
\textbf{Conclusion and further developments.} The Cox ring of an algebraic variety $Z$ is, in general, difficult to describe explicitly. 
Indeed, even the problem of understanding if $\Cox(Z)$ is finitely generated or not may be challenging.
On the other end, upper cluster algebras are very concrete objects whose underling combinatorial and algebraic structure is extremely rich and, in many aspects, well understood. 
Therefore, proving that the Cox ring of $Z$ is an upper cluster algebra, in an explicit way, can be a powerful tool for approaching its study and revealing unforeseen features.
For instance, the theory of quivers with potentials \cite{derksen2008quivers}, \cite{derksen2010quivers} can be used to develop a categorical and representation theoretic approach to the Cox ring.
Furthermore, one can obtain bases of the vector space $\Cox(Z)$, along with combinatorial models for the dimension of the spaces of global sections of line bundles of $Z$, through the well developed theory of good bases \cite{dupont2011generic}, \cite{qintriangular}, \cite{gross2018canonical},  \cite{qin2022bases}.
The strength of this approach can be seen in recent works on rings of semi-invariants of quivers \cite{fei2017cluster}, \cite{fei2017clusters}, and on branching algebras  \cite{magee2020littlewood}, \cite{fei2016tensor}.

Unfortunately, the current literature on cluster algebras does not yet provide the tools necessary to address the problem of the finite generation of the Cox ring.
Indeed, the problem of determining if a given upper cluster algebra is finitely generated, has not sufficiently been studied at the moment. 
It is well known that upper cluster algebras may be infinitely generated, as shown by Speyer \cite{speyer2013infinitely}.
Nevertheless, to my best knowledge, no general criteria for showing finite generation of an upper cluster algebra is known at the required level of generality.
I want to stress that the results of \cite{berenstein2005cluster3} and \cite{Muller2012locally} do not apply to the cluster algebras considered in this paper.
This is because the definition of cluster algebras that we use here allows frozen variables to be not invertible. 
A possible approach for generalizing the ideas of \cite{Muller2012locally} in a generality that covers the cluster algebras considered in this paper is provided by \cite{bucher2019upper}.

The existence of connections between cluster algebras and Cox rings were already present in early works on cluster algebras \cite{fomin2002clusterfoundation}, \cite{scott2006grassmannians}, \cite{geiss2008partial}, and became clear with the works of \cite{gross2013birational} and \cite{mandel2019cluster}. 
Nevertheless, the problem of identifying an explicit cluster structure on the Cox of ring of a given variety has never been studied in the generality of this paper. 
Note that, the problem of identifying an explicit cluster structure on a given ring is extremely difficult.
In general, this problem is highly sensible of the ring we are considering, and there is no general recipe for studying it.  
From this perspective, Theorem \ref{thm: intro 1} provides a precious tool for studying cluster structures on Cox rings.

\bigskip

We hope that the results of this paper could be used to reveal situations where the results of \cite{bossinger2024newton} can be applied.
Moreover, we hope that they could provide methods for studying the Cox rings of the partial compactifications of cluster varieties defined in \cite{gross2018canonical}.

\bigskip

As a final remark, I want to stress that in the present paper we show how to apply the minimal monomial lifting technique \cite{francone2023minimal} not only to Cox rings, but to the more general family of rings of global sections of sheaves of divisorial algebras.

\subsection{Structure of the paper}

Section \ref{sec:preliinaries Cox} starts by recalling the necessaries preliminaries on cluster algebras.
Then, we introduce the fundamental definitions and results on suitable for lifting schemes and minimal monomial lifting contained in \cite{francone2023minimal}.
Finally, we recall some facts on sheaves of line bundles, and their characteristic spaces, following \cite{arzhantsev2015cox}.
In Section \ref{sec:lifting clust structures} we show that the characteristic space of a sheaf of line bundles on a scheme admits many natural suitable for lifting structures, and we discuss how this simple observation allows to apply the minimal monomial lifting technique.
We mostly focus on the case of the Cox sheaf.
Then, we discuss some elementary examples, and we relate our construction to the results of \cite{cox1995homogeneous} on toric varieties.
Section \ref{sec: flag varieties} discusses the application to flag varieties and the relation with the previously discussed construction of \cite{geiss2008partial}.
In Section \ref{sec: diag comp cluster manifold}, we define the diagonal partial compactification of a finite cluster variety, and prove that its Cox ring has a cluster structure obtained by minimal monomial lifting.

\bigskip

In all the paper, we work over the base field of complex numbers. 
This is  mostly an expository choice, motivated by the fact that the results of \cite{francone2023minimal} are stated over $\CC$.
Despite this fact, all the results of this paper and of \cite{francone2023minimal} hold over any algebraically closed field of characteristic zero. 
 If a reader feels that the case of general fields deserves a
 publication, and has the patience to write down all the details, they will receive no contention.

\bigskip

\textbf{General notation.} The non-standard notation which is repeatedly used trough the text is introduced in the environments \textbf{"General notation"}.

\bigskip

\textbf{Warning.} Readers experienced with cluster algebras may skip most of the contents of Section \ref{Sec: cluster algebras reminders}, but they are still encouraged  to have a look at the notation and the definition of cluster and upper algebras.
Indeed, in this paper, we deal with cluster algebras that have both invertible and non-invertible coefficients (frozen variables).
Therefore, our definitions and notation differ slightly from the standard ones, in order to meet the need of varying coefficients.

\section{Preliminaries}
\label{sec:preliinaries Cox}

In all the paper, we work over the base field $\CC$ of complex numbers.
Therefore, notation as $\AAA^n$ and $\PP^n \, (n \in \NN)$ stand respectively for the complex affine space and the complex projective space of dimension $n$.
Similarly, $\GG_m$ is the complex multiplicative group of dimension one.

\subsection{Cluster and upper cluster algebras}
\label{Sec: cluster algebras reminders}

We recall the definition of cluster and upper cluster algebras of geometric type. 
We allow coefficient (frozen) variables to be both invertible and not.
We follow the notation of \cite{francone2023minimal}.

\bigskip

\noindent \textbf{General notation.}
The notation introduced here is mostly used for dealing with the combinatorial properties of cluster algebras.

Let $b \in \RR$.
We define 
$$
b^+:= \max\{b,0\}, \qquad b^-:= \max\{-b,0\}. $$

Let  $J$, $K$ be finite sets and $M \in \RR^{J \times K}$.
We refer to $M$ as a $J$ times $K$ matrix.
We denote by $M^\pm \in \RR_{\geq 0}^{J \times K}$  the matrix obtained by applying component-wise to $M$ the corresponding operation. 
For $k \in K$, the $k$\textit{-th column} of $M$ is denoted by 
$$
M_{\bullet k}:=M_{|J \times \{k\}} \in \RR^J.
$$ 
Let $H$ be an abelian group, $h \in H^J$, $m \in \ZZ^J$ and assume that the matrix $M$ has integer coefficients. 
Then, we denote
$$
h^m:= \sum m_j h_j \quad \text{and} \quad h^M :=(h^{M_\bullet k})_{k \in K} \in (H)^K .
$$
Given a ring $R$, the notation $R^\times$ stands for the set of invertible elements of $R$.
Note that, if $H= R^\times$, then $h^m$ is a (multiplicative) Laurent monomial in the elements $h_j \, (j \in I)$.
Let  $I$ be a third finite set and $N \in \ZZ^{K \times I}$.
Then, note that
$$
h^{M N}=(h^M)^N,
$$ 
where the product $ \ZZ^{J \times K} \times \ZZ^{K \times I} \longto \ZZ^{J \times I }$ corresponds to composition of morphisms in the canonical identification between $\ZZ^{J \times K}$ and $\Hom_\ZZ(\ZZ^K , \ZZ^J).$
The elements of the canonical basis of $\ZZ^K$ are denoted by $e_k \, (k \in K)$.

\subsubsection{Seeds}
The combinatorial information needed to define a cluster algebra is encoded in an object called seed. Let $\KK$ be a field extension of $\CC$.

   \begin{definition}  A \textit{seed} $t$ of $\KK$ is a collection $$(I_{uf}, I_{sf}, I_{hf}, B, x ) $$ consisting of:

   \begin{itemize}
       \item[--] Three disjoint finite sets $I_{uf}, I_{sf}, I_{hf}$. 
       An element of $I:= I_{uf} \sqcup I_{sf} \sqcup I_{hf}$ (resp. $I_f:= I_{hf} \sqcup I_{sf }$) is  a \textit{vertex} (resp. \textit{frozen} vertex) of $t$. 
       A vertex is respectively called \textit{unfrozen} (or \textit{mutable}), \textit{semi-frozen}, \textit{highly frozen} if it belongs to $I_{uf}, I_{sf}, I_{hf}.$ 
       \item[--] An integer matrix $B \in \ZZ^{I \times I_{uf}}$, such that its \textit{principal part} $B^\circ :=B_{| I_{uf} \times I_{uf}}$ is \textit{skew-symmetrizable}.
       That is: there exists $d_i \in \NN_{>0} $ ($i \in I_{uf}$), such that  the condition $d_ib_{ij} = - d_jb_{ji}$ holds for any $i,j \in I_{uf}$.
       We say that $B$ is the \textit{extended exchange matrix} of the seed $t$.
       \item[--] A transcendence basis $x=(x_i)_{i\in I} \in (\KK^\times)^I$ of $\KK$ over $\CC$.
       This is the \textit{extended cluster} of the seed $t$.
       The elements $x_i \, (i \in I)$ are called \textit{cluster variables} of $t$. 
       A variable is said to be \textit{unfrozen} (or \textit{mutable}) (resp. \textit{semi-frozen}, resp. \textit{highly-frozen}, resp. \textit{frozen}) if its corresponding vertex is unfrozen (resp. semi-frozen, resp. highly frozen, resp. frozen).
   \end{itemize}       
   \end{definition}

If a seed is denoted by $t$, we implicitly assume that its defining data are denoted as in the previous definition. 
If a seed is called $t^\bullet$, where $\bullet$ is any superscript, then we add a $\bullet$ superscript  to all the notation. 
For example, we write $t'=(I'_{uf}, I'_{sf}, I'_{hf}, B', x')$.
We use a similar conventions if a seed is denote by $t(\bullet)$, for some symbol $\bullet$.

\bigskip

\noindent \textbf{Graphical notation for seeds.}
The generalized exchange matrix $B$ of a seed $t$ can be equivalently represented using the \textit{valued quiver} $Q$ defined as follows.

\begin{itemize}
    \item [-] The vertices set of $Q$ is $I$. For $i \in I$, the corresponding vertex of $Q$ is pictured by a symbol $\bigcirc$ (resp. $\square$, resp. $\blacksquare$) if $i \in I_{uf}$  (resp. $I_{sf}$, resp.$I_{hf}$), and it is labeled by $i$.
\end{itemize}
Given two vertices $i$ and $j$ of $I$, we adopt the convention that $b_{ij}=0$ if both $i$ and $j$ are frozen, and that $b_{ij}=-b_{ji}$ if exactly one between $i$ and $j$ is frozen.
Note that, in the latter case, exactly one between $b_{ij}$ and $b_{ji}$ is defined as a coefficient of the generalized exchange matrix, according to which vertex is unfrozen. Then, the valued arrows of $Q$ are defined as follows.
\begin{itemize}
    \item[-]  There is an arrow between $i$ and $j$, pointing towards $j$, if and only if $b_{ij} > 0$. In this case, the arrow is labeled by: "$b_{ij}, - b_{ji}$". Moreover, if $b_{ij} = - b_{ji}$, for low values of $b_{ij}$, we may write $b_{ij}$ distinct arrows from $i$ to $j$ instead of a labeled arrow.
\end{itemize}
\begin{example}
    \label{ex label seed}
    Let $t=( \{1,2\}, \{3\}, \{4\}, B, x)$ be a seed whose generalized exchange matrix is $$ B:= \bmat 0 & 3 \\
    -1 & 0 \\
    0 & -2 \\
    0 & 1
    \emat,
    $$
    where the rows and columns of $B$ are labeled in the obvious way.
    The valued quiver representing the matrix $B$ is pictured on the left below.
    On its right, we have a simplified version of it.

    
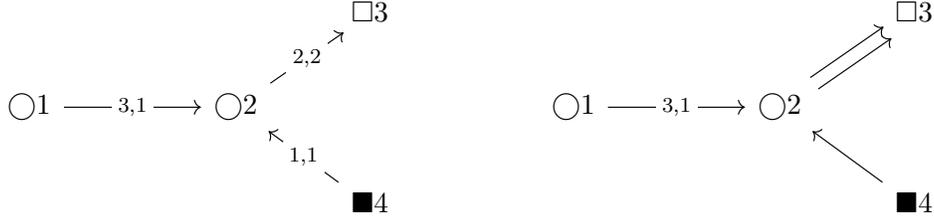
\[\begin{tikzcd}
	&&& {\square 3} &&&&& {\square 3} \\
	{\bigcirc 1} && {\bigcirc 2} &&& {\bigcirc 1} && {\bigcirc 2} \\
	&&& {\blacksquare 4} &&&&& {\blacksquare 4}
	\arrow["{2,2}"{description}, from=2-3, to=1-4]
	\arrow["{3,1}"{description}, from=2-1, to=2-3]
	\arrow["{3,1}"{description}, from=2-6, to=2-8]
	\arrow[shift left, from=2-8, to=1-9]
	\arrow[shift right, from=2-8, to=1-9]
	\arrow["{1,1}"{description}, from=3-4, to=2-3]
	\arrow[from=3-9, to=2-8]
\end{tikzcd}\]
    
\end{example}

\subsubsection{Mutations} Let $t$ be a seed of the field $\KK$ and $k \in I_{uf}$.
The mutation of $t$ at the vertex $k$ is the seed $\mu_k(t):= t'$ of $\KK$, where $t'=(I'_{uf}, I'_{sf}, I'_{hf}, B', x') $ is defined as follows. 
\begin{itemize}
    \item[--] 
    The vertices sets are $ \quad
    I_{uf}':=I_{uf}, \quad  I_{sf}':=I_{sf}, \quad I_{hf}':=I_{hf}.
    $
    \item[--] The extended exchange matrix $B' $ satisfies:

   \begin{equation}
       \label{matrix mutation}
       b'_{ij}:=\begin{aligned}
           \begin{cases} 
           -b_{ij} & \text{if} \quad i=k \quad \text{or} \quad j=k,\\
           b_{ij}+b_{ik}^+b_{kj}^+ - b_{ik}^-b_{kj}^- & \text{otherwise}.
           \end{cases}
       \end{aligned}
      \end{equation}
   \item[--] The extended cluster $x'$ is defined by
   \begin{equation}
       \label{cluster mutation}
      x_i':= \begin{aligned}
           \begin{cases} 
           x_i & \text{if} \quad i\neq k \\
           x \strut^{B_{\bullet k}^+ -e_k} + x \strut^{B_{\bullet k}^- - e_k} & \text{if} \quad i = k.
           \end{cases}
       \end{aligned}
      \end{equation}
\end{itemize}
The matrix $B'$ and the cluster $x'$ are also denoted by $\mu_k(B)$ and $\mu_k(x)$ respectively.
The identity 

\begin{equation}
    \label{exchange relation}
    x_kx_k'= x \strut^{B_{\bullet k}^+} + x \strut^{B_{\bullet k}^-}
\end{equation}
is called \textit{exchange relation}. We often denote the two monomials on the right hand side of the exchange relation by $M^+_k$ and $M_k^-$, in the obvious way.

Given a seed $t^*$ of $\KK$, we say that $t^*$ is \textit{mutation-equivalent} to $t$, and write $t^* \sim t$, if there exists $i_1, \dots , i_l \in I_{uf}$ such that 
$$
\mu_{i_l} \circ \dots \circ \mu_{i_1}(t)=t^*.
$$
The relation $\sim$ is an equivalence relation on the set of seeds of the field $\KK.$ 
The set of seeds mutation-equivalent to $t$ is denoted by $\Delta(t)$, or simply by $\Delta$ if there is no risk of confusion.

\subsubsection{Cluster and upper cluster algebras}
\label{sec: cluster upper cluster}
 For the rest of Section \ref{Sec: cluster algebras reminders}, we fix an equivalence class $\Delta$ of seeds of the field $\KK$.

 Since any seed in $\Delta$ has the same set of frozen variables, say $x_i$ for $i \in I_f$, and mutations preserve the type of frozen variables (semi-frozen or highly frozen),  the \textit{coefficient ring} 
\begin{equation}
 \label{coefficient ring}   
\coef[\Delta]:= \CC [x_i]_{i \in I_{hf}} [x_i^{\pm1}]_{i \in I_{sf}}
\end{equation}
is well defined.
We stress that Formula \eqref{coefficient ring} really depends on the decomposition $I_f=I_{sf} \sqcup I_{hf}$, and not only on the set $I_f$.
\begin{definition}
    \label{cluster algebra}
    The \textit{cluster algebra} $\clu(\Delta)$ is the $\coef[\Delta]$-subalgebra of $\KK$ generated by all the cluster variables of all the seeds in $\Delta$.
   A \textit{cluster variable} of $\clu(\Delta)$ is a cluster variable of some seed in $\Delta.$
\end{definition}

For a seed $t \in \Delta$, we define the ring of \textit{Laurent Polynomials}

\begin{equation}
    \label{Laurent Poly}
    \Li(t):= \coef[\Delta][x_i^{\pm1}]_{i \in I_{uf}}.
\end{equation}
Note that the unfrozen and semi-frozen variables of the seed $t$ are invertible in the ring $\Li(t)$, while the highly frozen ones are not.

\begin{definition}
The \textit{upper cluster algebra} $\uclu(\Delta)$ is  $$ \uclu(\Delta):= \bigcap_{t \in \Delta} \Li(t).$$
\end{definition}
Both rings $\clu(\Delta)$ and $\uclu(\Delta)$ are integral domains.
Moreover, $\uclu(\Delta)$ is normal, as it is defined as the intersection of normal rings.
\begin{remark}
\begin{enumerate}
    \item Let $t$ be a seed of $\Delta$.
    We also use the notation $\clu(t)$ (resp. $\uclu(t)$) to denote the algebra $\clu(\Delta)$ (resp. $\uclu(\Delta)$).
    If we use the latter notation, we consider the seed $t$ as the \textit{initial seed} of the cluster algebra $\clu(t)$. 
    Moreover, we refer to the cluster variables of $t$ as the \textit{initial} cluster variables.  
    \item  The choice of the ambient field $\KK$, as the choice of the cluster variables, is immaterial for the isomorphism class of a cluster or upper cluster algebra.
Indeed, these properties only depend on the generalized exchange matrix and on the set of vertices of an initial seed.
However, the choices of $\KK$ and of the cluster variables become important when we try to identify a given ring with an explicit cluster or upper cluster algebra.
\item The reader can find some examples of cluster and upper cluster algebras in Section \ref{sec: some examples}.
\end{enumerate}    
\end{remark}

We end this section by recalling some fundamental results on cluster and upper cluster algebras.

\begin{theo}[Laurent phenomenon]
\label{laurent pheno}
Let $t, t' \in \Delta$ and $i \in I$.
We have that 
$$
x_i' \in \ZZ[x_j]_{j \in I_f}[x_k^{\pm1}]_{k \in I_{uf}}.
$$
In particular, the cluster algebra $\clu(\Delta)$ is contained in the upper cluster algebra $ \uclu(\Delta)$.
\end{theo}

The Laurent Phenomenon is proved in \cite[Proposition 11.2]{fomin2002cluster}. (See also  \cite[Theorem 3.1]{fomin2002clusterfoundation}).

\begin{definition}
    \label{upper bound}
Let $t \in \Delta$. For $i \in I_{uf}$, let $t_i:= \mu_i(t)$ and $t_0:= t$. 
The \textit{upper bound} $\upp(t)$ of the seed $t$ is the algebra 
$$
\upp(t):= \bigcap_{i \in I_{uf} \cup \{0\} }\Li(t_i) .
$$
\end{definition}
 
If the matrix $B$ of a seed $t$ is of maximal rank, we say that the seed $t$ is of \textit{maximal rank}.
The property of being of maximal rank is invariant under mutation, by \cite[Lemma 3.2]{berenstein2005cluster3}. 

\begin{theo}
    \label{upper cluster equals upper bound}
    If the seed $t \in \Delta$ is of maximal rank, then $\uclu(\Delta)=\upp(t)$.
\end{theo}

Theorem \ref{upper cluster equals upper bound} is a very special case of \cite[Theorem 3.11]{gekhtman2018drinfeld}. 
The case with no highly frozen vertices had already been proved in \cite[Corollary 1.7]{berenstein2005cluster3}.

\begin{theo}
    \label{thm: inv upper}
    Assume that the seed $t \in \Delta$ is of maximal rank and let $i \in I_{uf} \cup I_{hf}$. 
    The cluster variable $x_i$ is an irreducible element of the ring $\uclu(\Delta)$. 
   Moreover, the invertible elements of the ring $\uclu(\Delta)$, up to complex scalars, are the monomials in the semi-frozen variables of any seed of $\Delta.$
\end{theo}
Theorem \ref{thm: inv upper} is essentially due to \cite{geiss2013factorial}.
Indeed, the proof of \cite[Theorem 1.3]{geiss2013factorial} also implies Theorem \ref{thm: inv upper}. 
See also  \cite[Theorem 2.35]{cao2022valuation}.

\subsubsection{Cluster valuations induced by frozen variables}
We denote  by $\CC(\Delta)$ the fraction field of $\uclu(\Delta)$.
Let $t\in \Delta$, and denote by $\CC[t]$ (resp. $\CC(t)$) the polynomial ring (resp. the field of fractions) in the cluster variables of the seed $t$.
By Theorem \ref{laurent pheno}, we have that $\CC[t]$ is contained in $\uclu(\Delta)$.
Moreover,  we have that $\CC(t)= \CC(\Delta).$


To any frozen vertex $i \in I_f$, we can associate the discrete divisorial valuation $\val_i^t$ corresponding to the vanishing locus of the frozen variable $x_i$ in 
the affine space $\Spec\bigl( \CC[t]\bigr).$ 
In other words,

$$
\val_i^t : \CC(t) \longto \ZZ \cup \{ \infty \}
$$
is the only discrete valuation that associates to a polynomial $f= \sum_{n \in \NN^I} a_n x^n$ $(a_n \in \CC)$ the integer 
$$
\val_i^t(f):= \min \{ n_i \, : \, n \in \NN^I \, \, \text{and} \, \, a_n \neq 0\}.
$$

By identifying the fields $\CC(t)$ and $\CC(\Delta)$ as described above, the valuation $\val_i^t$ induces a discrete valuation on the field $\CC(\Delta)$.

By \cite[Lemma 2.3.1]{francone2023minimal}, the resulting valuation on $\CC(\Delta)$ does not depend on the choice of the seed $t$.  
We call it the \textit{cluster valuation} associated to the frozen vertex $i$, and denote it by 
$$
\cval_i: \CC(\Delta) \longto \ZZ \cup \{ \infty \}.
$$

\subsubsection{Graded seeds}
We recall the notion of \textit{graded seed} and graded cluster algebra, formalized by \cite{grabowski2015graded}.

\begin{definition}
Let $t$  be a seed of $\Delta$ and $H$ be an abelian group. An element $\sigma \in H^I$ is an $H$-\textit{degree configuration}, on the seed $t$, if the condition
\begin{equation}
\label{graduation condition}
    \sigma^{B^+}= \sigma^{B^-}
\end{equation}
is fulfilled.
A seed with an $H$-degree configuration is called an $H$-\textit{graded seed}.
\end{definition}
  Condition \eqref{graduation condition} is equivalent to the fact that, for any $k \in I_{uf}$, then 
$$
\sum_{i \in I} b_{ik}^+ \sigma_i = \sum_{i \in I} b_{ik}^- \sigma_i.
$$



Assume that $t \in \Delta$ is an $H$-graded seed, and denote by $\sigma$ its degree configuration.
Let $k \in I_{uf}$ and $t':=\mu_k(t)$. 
Moreover, let  $\sigma'$ be the $H$-degree configuration on the seed $t'$ defined by the following formula:

\begin{equation}
  \label{mutation graduation}  
  \sigma'_i =\begin{aligned}
           \begin{cases} 
           \sigma_i & \text{if} \quad i \neq k \\
         \sigma^{B_{\bullet k}^+} - \sigma_k & \text{if} \quad i=k.
          \end{cases}
       \end{aligned}    
\end{equation}

We say that $\sigma'$ is the \textit{mutation at $k$} of the degree configuration $\sigma$, and we denote it by $\mu_k(\sigma)$.
If $k_1, \dots , k_m$ are mutable vertices of $t$, for some $m \in \NN$, the degree configuration $\mu_{k_1} \circ \cdots \circ \mu_{k_m}(\sigma)$ on the seed $\mu_{k_1} \circ \cdots \circ \mu_{k_m}(t)$ depends only on the seed $\mu_{k_1} \circ \cdots \circ \mu_{k_m}(t)$, and not on the mutation sequence.

Therefore, any seed $t^* \in \Delta$ is in a canonical way a graded seed.
We denote its degree configuration by $\sigma^*$.
Note that the algebra $\Li(t^*)$ is canonically $H$-graded by  setting $\deg(x^*_i):= \sigma^*_i$.
The set of Laurent monomials (in the variables of $x^*$) that belong to $\Li(t^*)$ forms an homogeneous base of  $\Li(t^*)$ as a $\CC$-vector space.
Moreover, we have that the algebras $\clu(\Delta)$ and $\uclu(\Delta)$ are graded subalgebras of $\Li(t^*)$, and the induced graduations on  $\clu(\Delta)$ and  $\uclu(\Delta)$ do not depend on the choice of the seed $t^*$, but only on the initial degree configuration $\sigma.$

\bigskip

    To sum up, the cluster and  upper cluster algebra of a graded seed are canonically graded algebras. Moreover, any cluster variable is an homogeneous element with respect to the corresponding graduation.

\subsection{Suitable for lifting schemes and minimal monomial lifting}
\label{sec: suitable lift schemes}

In this section we recall the fundamental definitions and results of \cite{francone2023minimal}.

\bigskip

\textbf{General notation.} If $H$ is a complex algebraic group, we denote by $X(H)$ its group of characters.
Assume that $H$ is a torus and a base $(\lambda_d)_{d \in D}$ of the free abelian group $X(H)$ has been fixed, for some finite set $D$.
Then, we identify the groups $\ZZ^D$ and $X(H)$ by means of the map assigning to an element $n \in \ZZ^D$ the character 
\begin{equation}
\label{eq: identification characters lattice}
    \varpi_n:= \sum_{d \in D} n_d \lambda_d .
\end{equation}
The torus $\GG_m^D$ is always considered with its base of characters corresponding to the natural coordinates of $\GG_m^D$.

\bigskip

Let $D$ be a finite set, $T$ be torus of dimension $|D|$ and $X(T)$ be its character group.
Let $\lX$ be a normal, noetherian, integral $\CC$-scheme acted on by $T$, and $X=(X_d)_{d \in D}$ be a collection of regular functions on $\lX$.
Let $\CC(\lX)$ be the field of rational functions on $\lX$.
We assume that, for any $d \in D$, the zero locus of $X_d$ in $\lX$ is irreducible, and we denote by $\val_d: \CC(\lX) \longto \ZZ \cup \{\infty\}$ its corresponding discrete divisorial valuation. 
Namely, the valuation $\val_d$ associates to a rational function on $\lX$ its order of zero or pole along  the vanishing locus of the function $X_d.$
Moreover, let $\phi : T \times Y \longto \lX$ be a $T$-equivariant open embedding. 
Here, $Y$ is some integral $\CC$-scheme, and $T$ acts on $T \times Y$ by multiplication on the $T$-component of $T \times Y.$  

\begin{definition}
\label{def: suitable lifting}
    The triple $(\lX, \phi, X)$ is a \textit{homogeneously suitable for $D$-lifting scheme} if the following conditions are satisfied.
    \begin{enumerate}
        \item For any $d_1,d_2 \in D$, then $\val_{d_1}(X_{d_2})= \delta_{d_1,d_2}.$
        \item We have that $\phi^*(X_d)=x_d \otimes 1$ $(d \in D)$ for some base $(x_d)_{d \in D}$ of the free abelian group $X(T)$. 
        \item The irreducible components of maximal dimension of $\lX \setminus \phi(T \times Y)$ are precisely the vanishing loci of the functions $X_d$, for $d \in D.$
    \end{enumerate}
\end{definition}

We assume from now on that the triple $(\lX, \phi, X)$ is a homogeneously suitable for $D$-lifting scheme. 
We fix below some notation and conventions that are freely used through the text when a homogeneously suitable for lifting scheme is fixed.

\bigskip

\noindent \textbf{General notation.}
Note that a base $(x_d)_{d \in D}$ of $X(T)$ as in Condition 2 of Definition \ref{def: suitable lifting} is uniquely determined by the data of $\phi$ and $X$. 
Therefore, the identification between $\ZZ^D$ and $X(T)$ given by formula \eqref{eq: identification characters lattice}, according to the base  $(x_d)_{d \in D}$, is freely used.

Moreover, we identify $T \times Y$ with its image in $\lX$ through the map $\phi$.
This allows to think about $T \times Y$ as an open subset of $\lX$, and write for example that $\CC(\lX)= \CC(T \times Y)$ and that $X_d= x_d \otimes 1$. 
In particular, for any $f \in \CC(Y)$, the notation $1 \otimes f$ identifies a well defined rational function on $\lX$.
Finally, we denote by $\iota: Y \longto \lX$ the map sending a point $y$ of $Y$ to $\iota(y):= (e, y) \in \lX$.

\bigskip

The algebra $\Orb_\lX(\lX)$ is $X(T)$-graded by means of the $T$-action on $\lX$. Namely, we have a decomposition
$$
\Orb_\lX(\lX) =\bigoplus_{\lambda \in X(T)} \Orb_\lX(\lX)_\lambda,
$$
where $\Orb_\lX(\lX)_\lambda \, (\lambda \in X(T))$ denotes the space of semi-invariant functions on $\lX$ of weight $\lambda.$ 
In other words, $\Orb_\lX(\lX)_\lambda$ is the vector space consisting of functions $f$ on $\lX$ that satisfy
$$
f(t \cdot x)= \lambda(t)f(x) \qquad (t \in T, \, x \in \lX).
$$

\begin{lemma}
    \label{lem: pole filtration}
    Let $n \in \ZZ^D$. 
    The restriction of $\iota^*$ to $\Orb_\lX(\lX)_{\varpi_n}$ is injective and its image is
    $$
    \Orb_Y(Y)_n:= \{ f \in \Orb(Y) \, : \, \val_d(1 \otimes f) \geq -n_d \quad \text{for any} \quad d \in D \}.
    $$
    The map $\Orb_Y(Y)_n \longto \Orb_\lX(\lX)_{\varpi_n}$ sending an element $f \in \Orb_Y(Y)_n $ to $X^n(1 \otimes f)$ is the inverse of the restriction of $\iota^*.$
\end{lemma}

\begin{proof}
   The direct proof of the statement is elementary. 
   The result is also an immediate consequence of \cite[Corollary 4.0.7, Proposition 3.1.6, Corollary 3.1.7]{francone2023minimal}.
\end{proof}
The vector spaces $\Orb_Y(Y)_n$ $(n \in \ZZ^D)$ are the components of a filtration on $\Orb_Y(Y)$.
The following corollary is an easy consequence of Lemma \ref{lem: pole filtration}.
\begin{coro}
        \label{cor:cox, to compare Leclerc}
    Let $f \in \Orb_Y(Y) \setminus \{0\}$. There exists a unique $n \in \ZZ^D$ and $\widetilde f \in \Orb_\lX(\lX)_{\varpi_n}$ with the following properties.
    \begin{enumerate}
        \item $\iota^*(\widetilde f)=f$.
        \item If $m \in \ZZ^D$ and there exists $\overline{f} \in \Orb_\lX(\lX)_{\varpi_m}$ such that $\iota^*(\overline{f})=f$, then $n \leq m$ (componentwise) and $\overline{f}= X^{(m-n)} \widetilde f $.
    \end{enumerate}
    Moreover, we have that 
    $$
    \widetilde f= X^n (1 \otimes f) \quad \text{and}  \quad n= \big( -\val_d(1 \otimes f)\big)_{d \in D}.
    $$
\end{coro}

\begin{proof}
    Let $n':= \big( -\val_d(1 \otimes f)\big)_{d \in D}.$ 
    Clearly, we have that $f \in \Orb_Y(Y)_{n'}$. 
    Then, Lemma \ref{lem: pole filtration} implies that
    $$
    X^{n'} (1 \otimes f) \in \Orb_\lX(\lX)_{\varpi_{n'}}.
    $$
   Let  $m \in \ZZ^D$ and $\overline{f} \in \Orb_\lX(\lX)_{\varpi_m}$ satisfying $\iota^* \big( \overline{f} \big)= f$. Then, Lemma \ref{lem: pole filtration} implies that 
    $$
    \overline{f}= X^m (1 \otimes f).
    $$
    Moreover, using the first condition of Definition \ref{def: suitable lifting} and the fact that $\overline{f}$ is regular on $\lX$, we deduce that 
    $$
    0 \leq \val_d \big( \overline{f} \big)= m_d -n_d' \qquad (d \in D).
    $$
    Therefore, $m$ is componentwise greater or equal then $n'$, and the statement follows.
\end{proof}

Assume from now that $t$ is a seed  such that $\Orb_Y(Y)= \uclu(t).$
As usual, we use the notation $t=(I_{uf}, I_{sf}, I_{hf}, B, x).$
\begin{definition}
\label{def: minimal lift matrix}
    The  matrix $\nu \in \ZZ^{D \times I}$ defined by 
    $
\nu_{d,i}:= -\val_d( 1 \otimes x_i) \, (d \in D, \, i \in I)
    $
    is called the \textit{minimal lifting matrix} of the seed $t$ with respect to $(\lX, \phi, X).$
\end{definition}

Starting from the minimal lifting matrix $\nu$, we can construct the seed    
$$
\lif t =( \lif I_{uf}, \lif I_{sf}, \lif I_{uf}, \lif B , \lif x), 
$$
of the field of fractions of $\Orb_\lX(\lX)$ (which is a subfield of $\CC(\lX)$), defined by the following data.

    \begin{itemize}
        \item[--] The sets of vertices  
        $$
        \lif I_{uf}:=I_{uf}, \quad \lif I_{sf}:=I_{sf}\sqcup D, \quad  \lif I_{hf}:= I_{hf}. 
        $$
        \item[--] The matrix $\lif B$ satisfying 
        $$
        \lif B_{|I \times I_{uf}}= B, \qquad \lif B_{|D \times I_{uf}}= -\nu B.
        $$
        For simplicity, we use the notation
        \begin{equation}
            \label{formula lifting B}
            \lif B= \begin{pmatrix}
            B \\
            -\nu B
        \end{pmatrix} \in \ZZ^{(I \sqcup D) \times I_{uf}}.
        \end{equation}
         The coefficients of $\lif B$ are denoted by $\lif b_{ij}$  $(i \in \, \lif I, j \in \, \lif I_{uf})$.
        
\item[--] The cluster $\lif x$ defined by
\begin{equation}
  \label{formula lifting variables}  
 \lif x_j= \begin{cases}
     \widetilde x_j & \text{if} \quad j \in I, \\
     X_j & \text{if} \quad j \in D,
 \end{cases}
 \end{equation}
 where, for any cluster variable $x_j$, the element $\widetilde x_j$ is defined by Corollary \ref{cor:cox, to compare Leclerc}.
\end{itemize}

\begin{remark}
    Thinking about $\nu$ (resp. $B$) as a matrix whose rows and columns are indexed by $D$ and $I$ (resp. $I$ and $I_{uf}$) respectively, it makes sense  to multiply $\nu$ with $B$.
We think about this product of matrices as an element $\nu B \in \ZZ^{D \times I_{uf}}$.
\end{remark}

\begin{definition}
    \label{def: minimal mon lifting}
   The \textit{minimal monomial lifting} of $t$ with respect to $(\lX, \phi, X)$ is the seed $\lif t^D$, of the field of fractions of $\Orb_\lX(\lX)$, obtained from $\lif t$ by switching the status of the vertices indexed by $D$ from semi-frozen to highly frozen. 
\end{definition}

Practically, the difference between the seeds $\lif t$ and $\lif t^D$ is given by the fact that the frozen cluster variables $\lif x_d$ ($d \in D)$ are non-invertible elements of the upper cluster algebra $\uclu(\lif t^D)$, while they are invertible in $\uclu(\lif t)$. 
Moreover, by \cite[Lemma 2.4.2]{francone2023minimal}, we have that 
\begin{equation}
\label{eq: localising clust min lifting}
   \uclu(\lif t^D)_{ \prod_{d \in D} \lif \spc x_d}= \uclu(\lif t).
\end{equation}

The seeds $\lif t$ and $\lif t^D$ are canonically $\ZZ^D$-graded by means of the degree configuration $\lif 0$ defined by

$$
\lif 0_j= \begin{cases}
    \nu_{\bullet j} & \text{if} \quad j \in I,\\
    e_j & \text{if} \quad j \in D,
\end{cases}
$$
where the $e_d$ ($d \in D$) are the canonical basis elements of $\ZZ^D$. 
Recalling the identification between $\ZZ^D$ and $X(T)$ induced by the Map \eqref{eq: identification characters lattice}, the previous degree configuration induces a $X(T)$-graduation on both $\uclu(\lif t^D)$ and $\uclu(\lif t).$

\bigskip

Recall that two regular functions $f,g $ on a noetherian scheme $Z$ are \textit{coprime}  if they don't simultaneously vanish on a codimension one closed subset of $Z$. 
The following statement is part of \cite[Theorem 4.0.3, Lemma 4.1.7]{francone2023minimal}.

\begin{theo}
\label{thm: min mon lifting}
Assume that the following conditions hols.
        \begin{enumerate}
            \item The seed $t$ is of maximal rank.
            \item For any non-equal $i,j \in I_{uf}$ the cluster variables $x_i$ and $x_j$ are coprime on $Y$.
            \item For any $k \in I_{uf}$,  the cluster variables $x_k$ and $\mu_k(x)_k$ are coprime on $Y$.
        \end{enumerate}
        Then, we have an inclusion $\uclu(\lif t^D) \subseteq \Orb_\lX(\lX)$ and an equality $\uclu(\lif t)= \Orb_\lX(T \times Y)$ of $X(T)$-graded algebras.
\end{theo}

Assumptions 2 and 3 of Theorem \ref{thm: min mon lifting} hold for example if $Y$ is affine and $\Orb_Y(Y)$ is factorial, because of \cite[Lemma 4.0.9]{francone2023minimal}.

Assume from now on that the hypothesis of Theorem \ref{thm: min mon lifting} are satisfied. 
Note that the seeds $t$ and $\lif t^D$ have the same set of mutable vertices.
Moreover, by \cite[Corollary 3.0.14]{francone2023minimal}, the map 
$$
\begin{array}{rcll}
     \Delta(t) & \longto & \Delta(\lif t^D) &   \\
    \mu_{i_1} \circ \cdots \circ \mu_{i_m}(t) & \longmapsto &  \mu_{i_1} \circ \cdots \circ \mu_{i_m}(\lif t^D) & (i_1, \dots , i_m \in I_{uf}, \, m \in \NN )
\end{array}
$$
is a well defined bijection. 
Furthermore, the restriction map $\iota^*$ sends the cluster variables of $\uclu(\lif t^D)$ (excluding the frozen variables $\lif x_d$, for $d \in D$) to cluster variables of $\uclu(t)$, and commutes with mutation. In formulae, we have that. 
\begin{equation}
    \label{eq: restriction clust variables}
    \iota^* \big(  \mu_{i_1} \circ \cdots \circ \mu_{i_m}(\lif x)_i \bigr)=  
      \mu_{i_1} \circ \cdots \circ \mu_{i_m}( x)_i \qquad (i_1, \dots , i_m \in I_{uf}, \, m \in \NN , \, i \in I).
\end{equation}
Note also that any cluster variable of $\uclu(\lif t^D)$ is a $X(T)$-homogeneous element of $\Orb(\lX)$.
The degree of the cluster variables can be computed by iteratively applying Formulae \eqref{mutation graduation}, starting from the initial degree configuration $\lif 0$ of the seed $\lif t^D.$

\bigskip

We say that a seed $\bar t= (\bar I_{uf}, \bar I_{sf}, \bar I_{hf}, \bar B, \bar x)$ of the fraction field of $\Orb_\lX(\lX)$ \textit{$D$-extends} $t$ if the following conditions are satisfied:
$$
\begin{array}{lll}
   I_{uf}= \bar I_{uf},  &  I_{sf} \subseteq \bar I_{sf}, &  I_{hf} \subseteq \bar I_{hf} \\[0.5 em]
   \bar I \setminus I= D,  & \bar B_{| I \times I_{uf}}=B, & \bar x_d = X_d \quad (d \in D). 
\end{array}
$$
The following theorem strongly motivates the interest in the seeds obtained by minimal monomial lifting (\cite[Theorem 4.1.9]{francone2023minimal}).

\begin{theo}
    \label{thm: uncity min mon lifting}
    Let $\bar t$ be a seed of the fraction field of $\Orb_\lX(\lX)$ that $D$-extends $t$, and assume that the following hold.
\begin{enumerate}
    \item $\uclu\bigl(\bar t\bigr)= \Orb_\lX(\lX)$.
    \item For any $i \in I$, the cluster variable $\bar x_i$ is a $X(T)$-homogeneous element of $\Orb_\lX(\lX)$ and $\iota^*(\bar x_i)=x_i.$
    \item  For any $k \in I_{uf}$, the cluster variable $\mu_k(\bar x)_k$ is a $X(T)$-homogeneous element of $\Orb_\lX(\lX)$ and $\iota^*\bigl( \mu_k(\bar x)_k \bigr)=\mu_k(x)_k.$
\end{enumerate}
Then $\bar t= \lif t^D$.
\end{theo}
In other words, the seed $\lif t^D$ provides the best possible candidate for giving the ring $\Orb_\lX(\lX)$ a cluster structure compatible with its $X(T)$-graduation and with the base  cluster structure on $\Orb_Y(Y)$ determined by the seed $t$.
Nevertheless, the inclusion between $\uclu(\lif t^D) $ and $\Orb_\lX(\lX)$ may be strict. 
The problem of determining if $\uclu(\lif t^D)$ equals $ \Orb_\lX(\lX)$ can be addressed, geometrically, as follows.

\bigskip

Note that, by Theorem \ref{thm: min mon lifting}, the rings $\uclu(\lif t^D)$ and $\Orb_\lX(\lX)$ have the same fraction field. 
Therefore, for any $d \in D$, we can evaluate the cluster valuation $\cval_d$ induced by the vertex $d$ of the seed $\lif t^D$ at elements of the ring $\Orb_\lX(\lX)$.
The next proposition is part of \cite[Proposition 4.1.4]{francone2023minimal}.

\begin{prop}
    \label{prop: equality conditions} The following are equivalent.
    \begin{enumerate}
      \item  $\uclu(\lif t^D)= \Orb_\lX(\lX).$
    \item For any $d \in D$, we have that $\cval_d\bigl(\Orb_\lX(\lX) \bigr) \subseteq \NN \cup \{ \infty \}$.
    \item For any $d \in D$, we have that $\cval_d = \val_d$ over $\Orb_\lX(\lX).$ 
    \end{enumerate}
\end{prop}

\subsection{Sheaves of line bundles and the Cox sheaf}

 Let $Z$ be an integral, normal, locally factorial $\CC$-scheme of finite type. Unless explicitly stated, the symbol $U$ refers to an open subset of $Z$.
 We start by recalling some classical results and notation, following \cite{hartshorne2013algebraic}.

We denote by $\Divv(Z)$  the group of \textit{Weil divisors} of $Z$. 
Recall that $\Divv(Z)$ is the free abelian group on  the set $Z^{(1)}$ of  closed irreducible subspaces of $Z$ codimension one.
The \textit{prime} divisors of $Z$ are the canonical basis elements of the group $\Divv(Z)$ corresponding to the set $Z^{(1)}$.
Recall that a divisor  $E= \sum_{S \in Z^{(1)}}n_S S$ is \textit{effective} if, for every $S \in Z^{(1)}$, we have that $n_S \geq 0$. 
In this case, we write $E \geq 0$, and we extend this notation to a partial order on $\Divv(Z)$ by saying that $E \geq E'$ if $E-E' \geq 0$. 
For a rational function $f \in \CC(Z)^\times$, we denote by  $\divv(f) \in \Divv(Z)$ the divisor of $f$.
For a divisor $E \in \Divv(Z)$, the associated quasi-coherent sheaf on $Z$, denoted by $\Orb(E)$, is defined by the formula
\begin{equation}
    \label{eq: sections O(E)}
\Orb(E)(U) := \{ f \in \CC(Z)^\times \, : \, \divv(f)_{|U} + E_{|U} \geq 0 \} \cup \{0\}.
\end{equation}
We denote by $\Pic(Z)$ the group of isomorphism classes of line bundles (or invertible sheaves) on $Z$.
Recall that, because of \cite[Proposition 6.11]{hartshorne2013algebraic}, any Weil divisor of $Z$ is a Cartier divisor. Therefore, we have an exact sequence
\[
\begin{tikzcd}[row sep=tiny]
	{\CC(Z)^\times} && {\Divv(Z)} && {\Pic(Z)} && 0 \\
	&& E && {[\Orb(E)],}
	\arrow["\divv", from=1-1, to=1-3]
	\arrow[from=1-3, to=1-5]
	\arrow[from=1-5, to=1-7]
	\arrow[maps to, from=2-3, to=2-5]
\end{tikzcd}
\]
where $[\Orb(E)]$ denotes the class of the sheaf $\Orb(E)$ in the group $\Pic(Z).$
Let $\eta$ be the generic point of $Z$. 
For a  quasi-coherent sheaf $\Fc$ on $Z$, an element of the stalk $\Fc_\eta$ is called a \textit{rational section} of $\Fc$.

\bigskip

Let $\Fc $ be a line bundle on $Z$.
A \textit{local parameter} $\sigma$ for the line bundle $\Fc$, on an open subset $U$ of $Z$, is a section $\sigma \in \Fc(U)$ such that the map $$\begin{array}{rcl}
   \Fc_{|U} & \longto & \Orb_U\\[0.3em]
   \tau & \longmapsto & \displaystyle \frac{\tau}{\sigma}
\end{array}$$
is a well defined isomorphism.
A \textit{trivialisation} $\{(U_j, \sigma_j) \}_{j \in J}$ of $\Fc$ is an open cover $\{U_j\}_{j\in J }$ of $Z$ and a collection of local sections $\{ \sigma_j \}_{j \in J}, $ such that $\sigma_j$ is a local parameter for $\Fc$ on the open subset $U_j$.
Given such a trivialisation and a non-zero rational section $\tau \in \Fc_\eta \setminus \{0\}$, we can define $\divv_Z(\tau) \in \Divv(Z)$: the  \textit{divisor} of $\tau$, by 
\begin{equation}
    \label{eq: divisor of a section}
    \divv_Z(\tau)_{|U_j}:= \divv \biggl (\frac{\tau}{\sigma_j} \biggr )_{|U_j} \qquad (j \in J).
\end{equation}
It is easy to see that the divisor $\divv_Z(\tau)$ is well defined and that it is independent on the trivialisation. 
Given a non-zero rational section $\sigma \in \Fc_\eta$, the map 
\begin{equation}
\label{eq:line boundles and div}
    \begin{array}{rcl}
   \Fc & \longto & \Orb\big(\divv_Z(\sigma)\big)\\[0.3em]
   \tau & \longmapsto & \displaystyle \frac{\tau}{\sigma}
\end{array}
\end{equation}
is a well defined isomorphism.

\bigskip

Let $D$ be finite set and $L=(L_d)_{d \in D}$ be a collection of line bundles indexed by $D$. 
For $n \in \ZZ^D$, we use the notation  
$$
L^n:= \bigotimes_{d \in D} L_d^{n_d}.
$$
We define the \textit{sheaf of line bundles} $\Rc_L$ by the formula
\begin{equation}
    \label{eq: R L}
    \Rc_L(U):= \bigoplus_{n \in \ZZ^D} L^n(U).
\end{equation}
This is a quasi-coherent sheaf of $\ZZ^D$-graded $\Orb_Z$-algebras.
For $n \in \ZZ^D$, the sheaf $L^n$ is the homogeneous component of $\Rc_L$ of degree $n$.
The multiplication of homogeneous elements of the sheaf $\Rc_L$ is given, locally, by tensor product of sections. 
If we fix isomorphisms $L_d \simeq \Orb(E_d)$, for some divisors $E_d$ ($d \in D$), then we have induced isomorphisms
$$
L^n \simeq \Orb\bigl( \sum_{d \in D} n_d E_d \bigr) \qquad (n \in \ZZ^D).
$$
Then, the multiplication of homogeneous elements of the sheaf $\Rc_L$ corresponds, via the previous isomorphisms and Eq. \eqref{eq: sections O(E)}, to the multiplication of functions in the field $\CC(Z).$

\begin{definition}
    \label{def:cox sheaf} If $\Pic(Z)$ is freely-generated by the classes of the line bundles of $L$, then $\Rc_L$ is called the \textit{Cox sheaf} of $Z$. Its $\Pic(Z)$-graded ring of global sections $\Rc_L(Z)$ is called the \textit{Cox ring} of $Z$, and it is denoted by $\Cox(Z)$. The Cox ring and Cox sheaf are well defined up to isomorphism \cite[Section 4.1]{arzhantsev2015cox}.
\end{definition}

Our objective is to try to construct cluster algebra structures on the ring $\Rc_L(Z)$, for some pairs $(Z,L)$. We are mostly interested in the case of the Cox ring. 

\begin{remark}
    Let $Z$ and $L$ as before, and let $Z_{sm}$ be the smooth locus of $Z$.
The complement of $Z_{sm}$ in $Z$ is of codimension at least two.
Therefore, by \cite[Proposition 6.5]{hartshorne2013algebraic} and the Hartogs' lemma, we have natural isomorphisms induced by restriction
$$
\Pic(Z) \simeq \Pic(Z_{sm}), \qquad \Rc_L(Z) \simeq \Rc_L(Z_{sm}).
$$
Therefore, we can restrict our attention to smooth schemes. 
Note that we are not assuming $Z$ to be separated.
\end{remark}

\subsection{The characteristic space}
\label{sec:charact space}
Let $Z$, $D$ and $L$ be as in the previous section. 
We assume furthermore that $Z$ is smooth.
We denote by 
\begin{equation}
\label{eq: def XL rel spctrum}
    \lX_L:= \Specbf(\Rc_L)
\end{equation}
the \textit{relative spectrum} of the sheaf of $\Orb_Z$-algebras $\Rc_L$, endowed with its canonical morphism $p: \lX_L \longto Z$, as defined for example in \cite{hartshorne2013algebraic}. 
The scheme $\lX_L$ and the morphism $p$ can be defined via a glueing construction. 
In particular, the morphism $p$ is affine and satisfies the property that
\begin{equation}
    \label{eq: sections charact space}
    p_*\Orb_{\lX_L}= \Rc_L.
\end{equation} 
Moreover, for any affine open subset $U$ of $ Z$, the morphism of affine schemes $p\inv(U) \longto U$ corresponds, at the level of regular functions, to the canonical inclusion map
$$
\Orb_Z(U)= \Rc_L(U)_0 \longto \Rc_L(U).
$$

The scheme $\lX_L$ and the morphism $p$ can also be characterized by means of the following universal property.
For any scheme $V$ and any pair $(f,\varphi)$ consisting of:
\begin{enumerate}
    \item a morphism $f:V \longto Z$,
    \item a  morphism of $\Orb_Z$-algebras $\varphi : \Rc_L \longto f_* \Orb_V$,
\end{enumerate}
 there exists a unique morphism $ \widetilde f : V \longto \lX_L$ such that the following diagrams are commutative

\begin{equation}
    \label{eq:univ property rel spectrum}
    \begin{tikzcd}
	& {\lX_L} && {p_*\Orb_{\lX_L}} && {p_*\widetilde f_* \Orb_V} \\
	V & Z && {\Rc_L} && {f_*\Orb_V}
	\arrow["{\widetilde f}", from=2-1, to=1-2]
	\arrow["p", from=1-2, to=2-2]
	\arrow["f"', from=2-1, to=2-2]
	\arrow["{p_* \widetilde f^\sharp}"', from=1-4, to=1-6]
	\arrow[shift left, no head, from=1-4, to=2-4]
	\arrow[no head, from=1-4, to=2-4]
	\arrow["\varphi", from=2-4, to=2-6]
	\arrow[shift left, no head, from=1-6, to=2-6]
	\arrow[no head, from=1-6, to=2-6]
\end{tikzcd}
\end{equation}

Let $T:=\GG_m^D$.
Recall the identification between the groups $\ZZ^D$ and $X(T)$, that assigns to an element $n \in \ZZ^D$ the character $\varpi_n$ defined by Formula \eqref{eq: identification characters lattice}. 
The $\ZZ^D$-graduation on the sheaf $\Rc_L$ corresponds to a $T$-action on the scheme $\lX_L$, regarding to which the morphism $p$ is a good quotient. 
A proof of this fact can be found in \cite{arzhantsev2015cox}. 
The semi-invariant regular functions on $\lX_L$, with respect to this $T$-action, are the homogeneous global sections of the sheaf $\Rc_L$. 
In particular, we have that
\begin{equation}
    \label{eq: semi-inv characteristic space}
    \Orb_{\lX_L}(\lX_L)_{\varpi_n}= L^n(Z) \qquad (n \in \ZZ^D).
\end{equation}

As $Z$ is smooth, the map $p$ is actually a locally-trivial $T$-bundle in the Zariski topology. 
We can construct explicit trivializations as follows.

Let $\sigma = (\sigma_d)_{d \in D}$ be a collection of sections such that $\sigma_d \in L_d(Z)$. 
For $d \in D$, we denote by $Z_{\sigma_d}$ the non-vanishing locus of $\sigma_d$ and we use the notation
$$
Z_\sigma:= \bigcap_{d \in D} Z_{\sigma_d}.
$$
Note that the section $\sigma_d$ is a local parameter for the bundle $L_d$ on the open subset $Z_{\sigma_d}$.
Hence, using \eqref{eq:line boundles and div}, we have that for any affine open subset $U$ of $Z_\sigma$, the morphism
\begin{equation}
    \label{eq:local triv ssheaf of line b}
    \begin{array}{rcl}
       \Rc_L(U) & \longto & \CC[T] \otimes \Orb_Z(U)\\[0.3em]
       (\tau_n)_{n \in \ZZ^D} & \longmapsto & \displaystyle \sum_{n \in \ZZ^D} \varpi_n \otimes \frac{\tau_n}{\sigma^n}
    \end{array} 
\end{equation}
is a $\ZZ^D$-graded isomorphism. Therefore, the previous map induces a $T$-equivariant isomorphism $T \times U \longto p\inv(U)$, whose pullback of regular functions is defined by Formula \ref{eq:local triv ssheaf of line b}. 
It follows, by glueing, that we actually have a $T$-equivariant isomorphism
$$
T \times Z_\sigma \longto p\inv(Z_\sigma),
$$
locally defined by \eqref{eq:local triv ssheaf of line b}.
The previous discussion implies that the scheme $\lX_L$ is smooth and integral. In particular, the ring $\Rc_L(Z)=\Orb_{\lX_L}(\lX_L)$ is a normal domain.

\bigskip

Let $U$ be an open subset of $Z$ and $\tau $ be an element of $ L^n(U)$, for some $n \in \ZZ^D$. 
Recall the definition (Eq. \eqref{eq: divisor of a section}) of the divisor
 $\divv_Z(\tau) \in \Divv(Z)$, obtained by considering $\tau$ as a rational section of the line bundle $L^n$.
We write $\divv_{\lX_L}(\tau)$ for the divisor of $\lX_L$ obtained by considering $\tau$ as a rational function on $\lX_L$. 
We have the following relation \cite[Lemma 3.3.2]{arzhantsev2015cox} between these two divisors.

\begin{lemma}
    \label{lem:divisors rel spec}
Let $\tau \in L^n(U)$, for some $n \in \ZZ^D$. We have that $\divv_{\lX_L}(\tau)= p^*\big(\divv_Z(\tau)\big).$
\end{lemma}

The notation $p^*\big(\divv_Z(\tau)\big)$ stands for the pullback of $\divv_Z(\tau)$ along $p$, which is well defined since $\divv_Z(\tau)$ is a Cartier divisor, as $Z$ is locally factorial.

\section{Lifting cluster structures}
\label{sec:lifting clust structures}
\subsection{Suitable for lifting characteristic spaces}
\label{sec: Suitable for lifting ch.space}

Let $Z$, $D$, $L$, $T$ and $
p: \lX_L:=\Specbf(\Rc_L) \longto Z
$ as in Section \ref{sec:charact space}.
Moreover, let $\sigma=(\sigma_d)_{d \in D}$ be a collection of global sections such that $\sigma_d \in L_d(Z)$ and the corresponding divisors $E_d:= \divv_Z(\sigma_d) \, (d \in D)$ are prime and pairwise distinct.
By Lemma \ref{lem:divisors rel spec} and the fact that $p: \lX_L \longto Z$ is a locally trivial $T$-bundle in the Zariski topology, we deduce that the divisors $\divv_{\lX_L}(\sigma_d)= p^* \big( \divv_Z(\sigma_d) \big)$ are prime and pairwise distinct.
For $d \in D$, we denote by  
$$
\val_d: \CC(\lX_L) \longto \ZZ \cup \{\infty\}
$$
the valuation associated to the prime divisor $\divv_{\lX_L}(\sigma_d)$. 
 Let $Y:= Z_\sigma$. 
 Moreover, consider the $T$-equivariant isomorphism
 $$
 \phi: T \times Y \longto p\inv(Y)
 $$
 locally defined  by \eqref{eq:local triv ssheaf of line b}.

\begin{prop}
\label{prop:Line bund suitable lift}
    The triple $(\lX_L, \phi, \sigma)$ is a homogeneously suitable for $D$-lifting scheme.
\end{prop}

\begin{proof}
    Condition 1 of Definition \ref{def: suitable lifting} holds since the divisors $\divv_{\lX_L}(\sigma_d)$ are prime and pairwise different. 
    Equation \eqref{eq:local triv ssheaf of line b} implies that $\phi^*(\sigma_d)= \varpi_{e_d} \otimes 1$, hence Condition  2 holds.
    Condition 3 is clear.
    \end{proof}

If $Y$ has a cluster structure, Theorem \ref{thm: min mon lifting} allows, under mild hypothesis, to lift this cluster structure inside the graded ring $\Rc_L(Z)=\Orb_{\lX_L}(\lX_L)$. 
This procedure actually depends on the choice of $\sigma$.
If $\Pic(Z)$ is free and finitely generated and we want to construct cluster structures on the Cox ring of $Z$ via the characteristic space and minimal monomial lifting, we can make the choice of $L$ and $\sigma$ more canonical, whenever we find an appropriate open subset of $Z$. 
This is the content of the  next section.

\subsection{The case of the Cox ring}
\label{sec: case of Cox}

\textbf{Assumptions.} In this section, $Z$ is a smooth irreducible $\CC$-scheme of finite type. 
Moreover,  $Y$ is an  open subset of $Z$ satisfying the following assumptions.
\begin{enumerate}
    \item $\Pic(Y)=\{0\}.$ 
    \item $\Orb_Y(Y)^\times= \CC^\times.$
    \item Any irreducible component of $Z \setminus Y$ is of codimension one in $Z$.
\end{enumerate}

\begin{remark}
\label{rem: assumptions and invertible}
\begin{enumerate}
Let $Y$ be an open subset of $Z$. Then, the following hold.
    \item  If $Y$ is affine, then Assumption 3 is satisfied.
    \item  If $Y$ is affine, its Picard group is trivial if and only if the ring $\Orb_Y(Y)$ is factorial, because of \cite[Propositions 6.2,6.11]{hartshorne2013algebraic}.
    \item Assume that $\Orb_Y(Y)=\uclu(t)$, for some seed $t$ of maximal rank. 
     Then, the condition $\Orb_Y(Y)= \CC^\times$ holds if and only if the seed $t$ doesn't have any semi-frozen variable, because of Theorem \ref{thm: inv upper}.    
    \item Assumption 3 is technical and it is not restrictive as long as we are interested in rings of global sections of line bundles.
Indeed, suppose that $Y$ satisfies Assumptions 1 and 2 and let $D_1, \dots, D_n$  be the irreducible components of $Z \setminus Y$ of codimension strictly greater than one.
    Then, consider the scheme $Z':=Z \setminus \bigcup_{i=1}^n D_n$. 
    The restriction map induces canonical isomorphisms 
$$
\Pic(Z) \longto \Pic(Z') \qquad \Li(Z) \longto \Li(Z'),
$$
for any line bundle $\Li$ on $Z$.
Moreover, the set $Y$ satisfies all the assumptions, as a subset of $Z'.$
\end{enumerate}    
\end{remark}

 Let $D$ be the set of irreducible components of $Z \setminus Y$.
 To avoid confusion, for $d \in D$, we denote by $E_d \in \Divv(Z)$ the prime divisor corresponding to  $d$.  

\begin{lemma}
\label{lem:pic(Z) with nice open}
The group $\Pic(Z)$ is freely generated by the classes of the line bundles $\Orb(E_d)$.    
\end{lemma}

\begin{proof}
   Since $Z$ is smooth, by \cite[Proposition 6.5]{hartshorne2013algebraic} we have an exact sequence 
   $$ 
   \bigoplus_{d \in D} \ZZ E_d \longto \Pic(Z) \longto \Pic(Y) \longto 0.
   $$
    Assume that a divisor $E \in \bigoplus_{d \in D} \ZZ E_d$ is principal on $Z$.
    Then, there exists $f \in \CC(Z)^\times$ such that $\divv(f)=E.$
   We deduce that $f \in \Orb_Y(Y)^\times$, as $\divv(f)_{|Y}=0.$ 
   Hence, the function $f$ is constant, which implies that $E=0$. 
   As $\Pic(Y)= \{0\}$, it follows that the natural map $\bigoplus_{d \in D} \ZZ E_d \longto \Pic(Z) $ is an isomorphism.
\end{proof}

Let $L:=\big(\Orb(E_d)\big)_{d \in D}$ and, for $d \in D$, let $\sigma_d$ be the rational function $1$ seen as a global section of the line bundle $\Orb(E_d)$.
Moreover, as in Section \ref{sec:charact space}, let $T:= \GG_m^D$ and consider the isomorphisms
\begin{equation}
    \label{eq: X(T) e Pic(X)}
\begin{tikzcd}[row sep=tiny]
	{\Pic(Z)} && {\ZZ^D} & {} & {X(T)} \\
	{\bigl[L^n\bigr]} && n && {\varpi_n.}
	\arrow[from=1-3, to=1-1]
	\arrow[from=1-3, to=1-5]
	\arrow[maps to, from=2-3, to=2-1]
	\arrow[maps to, from=2-3, to=2-5]
\end{tikzcd}
\end{equation}
Note that $\divv_Z(\sigma_d)=E_d$.
In particular, we have that 
$$
Y=Z_\sigma.
$$
By Lemma \ref{lem:pic(Z) with nice open}, we have that $\Rc_L$ is the Cox sheaf of $Z$. 
As usual, let $p: \lX_L:= \Specbf(\Rc_L) \longto Z$ be the canonical morphism, and recall that we have an equality of $\Pic(Z)$-graded algebras $\Orb_{\lX_L}(\lX_L)= \Cox(Z)$.
Similarly to Section \ref{sec: Suitable for lifting ch.space}, we denote by 
\begin{equation}
    \label{eq: phi standard suitable}
    \phi: T \times Y \longto p\inv(Y)
\end{equation}
  the $T$-equivariant open embedding locally defined  by the collection of sections $\sigma$ and Eq. \eqref{eq:local triv ssheaf of line b}.

\begin{definition}
    \label{def:standard suitable lifting Cox} The \textit{standard homogeneously suitable for lifting structure of $\Cox(Z)$ corresponding to the open subset} $Y$ is the triple $(\lX_L, L, \sigma)$.
\end{definition}

\begin{remark}
\label{eq: terminology standard suitable cox}
 The triple $(\lX_L, L, \sigma)$ is not  a homogeneously suitable for lifting scheme in the sense of Definition \ref{def: suitable lifting}. 
   Conversely, the triple $(\lX_L, \phi, \sigma)$ is a homogeneously suitable for $D$-lifting scheme because of Proposition \ref{prop:Line bund suitable lift}.
  In the rest of the paper, we  commit the abuse of notation of using expressions as:
    \begin{center}
        \textit{the minimal monomial lifting of a seed with respect to $(\lX_L, L, \sigma)$} 
    \end{center}
   instead of 
   \begin{center}
       \textit{the minimal monomial lifting of a seed with respect to $(\lX_L, \phi, \sigma)$},
   \end{center}
   because we find it more convenient to track  the collection of line bundles $L$ rather than the map $\phi$.
       Note also that the subset $Y$ uniquely determines the data of $ L, \, \lX_L, \, \sigma, \, D$ and the map $\phi$.
\end{remark}

\begin{theo}
\label{thm:min lifting Cox}
    Assume that $\Orb_Y(Y)= \uclu(t)$ for some seed $t$ satisfying Hypothesis 1,2,3 of Theorem \ref{thm: min mon lifting}.
    Let $\lif t^D$ be the minimal monomial lifting of the seed $t$ with respect to  $(\lX_L, L, \sigma)$, considered as a $\Pic(Z)$-graded seed by means of the degree configuration $\lif 0$ and  the isomorphisms in \eqref{eq: X(T) e Pic(X)}.
 Then, as $\Pic(Z)$-graded algebras, we have that 
 $$
 \uclu(\lif t^D) \subseteq \Cox(Z), \qquad \uclu(\lif t^D)_{\prod_{d \in D} \sigma_d}= \Cox(Z)_{\prod_{d \in D} \sigma_d}.
 $$ 
\end{theo}

\begin{proof}
    Recall that $\Orb_{\lX_L}(\lX_L)= \Cox(Z)$. 
Moreover, note that the non vanishing locus of the function $\prod_{d \in D} \sigma_d$ in $\lX_L$ is  $p\inv (Y)$. 
    Then, since $\lX$ is of finite type, the natural map 
    $$
    \Orb_{\lX_L}(\lX_L)_{\prod_{d \in D} \sigma_d} \longto \Orb_{\lX_L}\bigl(p\inv(Y) \bigr) 
    $$
    is an isomorphism. 
    Therefore, using Eq. \eqref{eq: localising clust min lifting}, Theorem \ref{thm: min mon lifting} allows to conclude.
\end{proof}

\begin{lemma}
\label{lem: min mon cox affine}
   Assume that $Y$ is affine and $\Orb_Y(Y)= \uclu(t)$ for some seed $t$ of maximal rank. 
   Then, Hypothesis 2 and 3 of Theorem \ref{thm: min mon lifting} are satisfied, and the conclusion of Theorem \ref{thm:min lifting Cox} holds.
\end{lemma}
\begin{proof}
    Since $Y$ is smooth, affine, and $\Pic(Y)= \{0\}$, it follows that the ring $\Orb_Y(Y)$ is factorial.
   Then, the statement follows from \cite[Lemma 4.0.9]{francone2023minimal}.
\end{proof}

We conclude this section by explicitly describing the cluster variables of the seed $\lif t^D$ as sections of line bundles on $Z$, defined in terms of geometric properties of $Z$. 
Note that, by the general minimal monomial lifting construction, the seed $\lif t^D$ is naturally defined in terms of geometric properties of the scheme $\lX_L$.

\bigskip

Recall that, according to the general notation introduced in Section \ref{sec: suitable lift schemes}, we identify the space $T \times Y$ with its image through the map $\phi$ \eqref{eq: phi standard suitable}, and we have an immersion $\iota : Y \longto \lX_L$. 
Moreover, for any $n \in \ZZ^D$, we have that 
\begin{center}
$
\Orb_\lX(\lX)_{\varpi_n} = \HH^0 \bigl(Z, \, \Orb\bigl( \sum_{d \in D} n_d E_d \bigr) \bigr).
$
\end{center}
Finally, let $\val_{E_d}: \CC(Z) \longto \ZZ \cup \{\infty\} \, (d \in D)$ be the discrete divisorial valuation associated to the prime divisor $E_d$.

\begin{lemma}
\label{lem: homogenisation Cox}
Let $f \in \Orb_Y(Y) \setminus \{0\}$, and consider the element $n \in \ZZ^D$ and the semi-invariant function $\widetilde f$ on $\lX_L$ defined by Corollary \ref{cor:cox, to compare Leclerc}. 
We have that
\begin{center}
$
n= \big(- \val_{E_d}(f)\bigr)_{d \in D}, \qquad \widetilde f = f \in \HH^0 \bigl( Z, \, \Orb \bigl( \sum_{d \in D} n_d E_d \bigr) \bigr).
$
\end{center}
\end{lemma}

\begin{proof}
By Corollary \ref{cor:cox, to compare Leclerc}, we have that $n= \bigl( - \val_{p^*(E_d)}(1 \otimes f )\bigr)_{d \in D}$.
Note that the diagram
\[\begin{tikzcd}[row sep=scriptsize]
	{T \times Y} && {\lX_L} \\
	Y && Z
	\arrow[from=1-1, to=1-3]
	\arrow[from=1-1, to=2-1]
	\arrow["p", from=1-3, to=2-3]
	\arrow[from=2-1, to=2-3]
\end{tikzcd}\]
is commutative. 
The maps $T \times Y \longto \lX_L$, $Y \longto Z$, and $T \times Y \longto Y$ are  respectively the natural inclusions and projection.
Therefore, the equality $1 \otimes f= p^*(f)$, of rational functions on $\lX_L$, holds.
Hence, Lemma \ref{lem:divisors rel spec} implies the desired expression for $n.$

 Then, observe that since the rational function $f$ is regular on $Y$, we have that 
    $$
    \divv_Z(f) + \sum_{d \in D} - \val_{E_d}(f) \geq 0.
    $$
    Therefore, the rational function $f$ is a global section of the line bundle $\Orb\bigl( \sum_{d \in D} n_dE_d \bigr)$.
    We denote this section by $\overline{f}.$
    
    Let $m \in \ZZ^D$ and $\psi \in \HH^0 \bigl( Z, \, \Orb \bigl( \sum_{d \in D} n_d E_d \bigr) \bigr) \setminus \{0\}.$
    In other words, $\psi$ is a non-zero rational function on $Z$ satisfying the property that 
    $
    \divv_Z(\psi) + \sum_{d \in D}m_d E_d \geq 0.
    $
    In particular, the function $\psi$ is regular on $Y$, since $\divv(\psi)_{|Y} \geq 0.$
   Since  $\sigma_d \, (d \in D)$ is the rational function 1 considered as a global section of the line bundle $\Orb(E_d)$, it follows at once from Eq. \eqref{eq:local triv ssheaf of line b} and the definition of the map $\iota$ that $\iota^*(\psi)=\psi.$
   
    We deduce that $\iota^*\bigl( \overline{f} \bigr)=f$.
   In particular, $\overline{f}$ and $\widetilde f$ are two semi-invariant functions on $\lX_L$, of the same degree, whose pullback along $\iota$ is $f$. 
   Lemma \ref{lem: pole filtration} implies that $\overline{f}= \widetilde f.$
\end{proof}

\begin{coro}
    Let $\nu$ be the minimal lifting matrix of the seed $t$ with respect $(\lX_L, L, \sigma)$.
    We have that
    $$
    \nu_{di}= -\val_{E_d}(x_i) \qquad (d \in D, \, i \in I).
    $$
    Moreover, for $i \in \lif I= I \sqcup D$, we have that 
    $$
    \lif x_i = \begin{cases}
        x_i \in \HH^0 \bigl( Z, \,  \Orb \bigl( \sum_{d \in D} \nu_{di} E_d \bigr) \bigr) & \text{if} \quad i \in I \\[0.5em]
        1 \in \HH^0\bigl( Z, \, \Orb(E_d) \bigr) & \text{if} \quad i = d \in D.
    \end{cases}
    $$
\end{coro}

\begin{proof}
    This is a direct consequence of Lemma \ref{lem: homogenisation Cox} and the definition of the seed $\lif t^D.$
\end{proof}

\subsection{Some Examples}
\label{sec: some examples}
We consider some applications of the previous constructions, that also provide some instructive explicit examples. 
\subsubsection{The projective space}

Let $\PP^n$ be the complex projective space of dimension $n \in \NN$, with homogeneous coordinates $[Z_0 : \dots : Z_n]$. 
Let $Y \subseteq \PP^n$ be the open subset defined by the non-vanishing of $Z_0$, and set $E_0:= \divv(Z_0)$. 
The open set $Y$ is an affine space of dimension $n$, with coordinates 
$$
z_i:= \displaystyle \frac{Z_i}{Z_0} \qquad (1 \leq i \leq n).
$$
The pair $(Y , \PP^n)$ satisfies the assumptions of Section \ref{sec: case of Cox}. We clearly have that $(\PP^n, \Orb(1), Z_0)$ is the standard homogeneously suitable for lifting structure of $\Cox(\PP^n)$ corresponding to the open subset $Y$. In this case, we have that $D= \{0\}.$

The Cox sheaf and the Cox ring of $\PP^n$ are respectively 
$$
\Rc_{\Orb(1)}= \bigoplus_{k \in \ZZ} \Orb(k) \quad \text{and} \quad 
\Cox(Z)= \bigoplus_{k \in \NN } \CC[Z_0, \dots, Z_n]_k, 
$$
where $\CC[Z_0, \dots, Z_n]_k$ denotes the space of homogeneous polynomials of degree $k.$
The morphism $p: \lX_{ \Orb(1)}= \Specbf(\Rc_{\Orb(1)}) \longto \PP^n$ is the tautological $\GG_m$-bundle $\AAA^{n+1} \setminus \{0\} \longto \PP^n.$
If $f \in \Orb(Y)= \CC[z_1, \dots , z_n]$ is a polynomial of degree $m \in \NN$, then
$$
\val_{E_0}(f)= - m \quad \text{and} \quad  
\widetilde f= Z_0^m f \biggl( \displaystyle \frac{Z_1}{Z_0}, \dots , \displaystyle \frac{Z_1}{Z_0} \biggr),
$$
where $\widetilde f$ is defined by Corollary  \ref{cor:cox, to compare Leclerc}. 
The homogeneous polynomial $\widetilde f$ is the well known homogenization of the polynomial $f$ with respect to the variable $Z_0$.

\begin{example}
\label{ex: elementary proj}
    The affine space $Y$ has a trivial cluster structure associated to the seed $t$ with $n$ highly frozen cluster variables $z_1, \dots ,z_n$ and no other variable. 
    That is: $\uclu(t)= \CC[z_1 , \dots , z_n]$.
    For any $1 \leq i \leq n$, we have that $\val_{E_0}(z_i)=-1$.
    Hence, the minimal lifting matrix of the seed $t$ with respect to $(\PP^n, \Orb(1), Z_0)$ is 
    $$
    \nu= (1, \dots , 1),
    $$
    with obvious notation. The minimal monomial lifting, in this case, is the seed $\lif t^0$ which has $(n+1)$-frozen variables $Z_0, Z_1, \dots, Z_n$.
    In particular, we have that 
    $$
    \uclu(\lif t^0)= \Orb_{\lX_{\Orb(1)}}(\lX_{\Orb(1)})= \Cox(\PP^n).
    $$
\end{example}

\begin{example}
    This less trivial example is inspired from \cite[Section 7.1]{geiss2013factorial}. Consider the seed $t$ of the field $\CC(z_1, \dots , z_n)$ graphically represented by the following quiver 
    \[\begin{tikzcd}
	{ \bigcirc 1} & { \bigcirc 2} & \cdots & { \bigcirc (n-1)} & { \blacksquare n},
	\arrow[from=1-2, to=1-1]
	\arrow[from=1-3, to=1-2]
	\arrow[from=1-4, to=1-3]
	\arrow[from=1-5, to=1-4]
\end{tikzcd}\]
and with cluster variables $x=(x_1, \dots , x_n)$ defined as follows. We introduce the auxiliary notation $x_{-1}:= 0$ and $ x_{0}:= 1$, and  we set
$$
x_{k+1}:=z_{k+1}x_k-x_{k-1} \quad \text{for} \quad k \geq 0.
$$
     The element $x_{k} \, (1 \leq k \leq n)$ is a degree $k$-polynomial in $z_1, \dots , z_k$, and we have that $\uclu(t)= \CC[z_1, \dots, z_n].$
      The minimal lifting matrix of the seed $t$ with respect to $(\PP^n, \Orb(1), Z_0)$ is 
    $$
    \nu= (1, 2, \dots, n).
    $$
    By direct calculation, the seed $\lif t^0$ is graphically represented by the following quiver
    
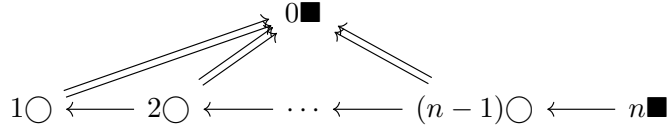
\[\begin{tikzcd}
	&& {0 \blacksquare } \\
	{1 \bigcirc} & {2 \bigcirc} & \cdots & {(n-1) \bigcirc} & {n \blacksquare}
	\arrow[from=2-2, to=2-1]
	\arrow[from=2-3, to=2-2]
	\arrow[from=2-4, to=2-3]
	\arrow[from=2-5, to=2-4]
	\arrow[shift left, from=2-1, to=1-3]
	\arrow[from=2-1, to=1-3]
	\arrow[from=2-2, to=1-3]
	\arrow[shift left, from=2-2, to=1-3]
	\arrow[from=2-4, to=1-3]
	\arrow[shift right, from=2-4, to=1-3]
\end{tikzcd}.\]
Moreover, we have that $\lif x_0=Z_0$ and $\lif x_1= Z_1$. 
In general, the variable $\lif x_k$ $(1 \leq k \leq n)$ is a homogeneous polynomial of degree $k$.
By direct calculation, we have that $\mu_k(\lif x)_k=Z_{k+1} \, (1 \leq k \leq n-1)$. 
In particular, we have that 
$$
\clu(\lif t^0)= \uclu(\lif t^0)= \CC[Z_0, Z_1, \dots, Z_n]= \Cox(\PP^n).
$$
\end{example}

\subsubsection{Toric varieties}

The following discussion generalises Example \ref{ex: elementary proj} in the context of toric varieties.

Let $Z$ be the toric variety determined by a fan $\Sigma$ in a lattice $N$ of rank $n$. 
We assume that $Z$ is smooth.
We recall that, by \cite[Section 2.1]{fultonToric}, the smoothness of $Z$ is equivalent to the fact that any cone of $\Sigma$ is generated by a subset of a base of the $\ZZ$-module $N$.
We denote by $\Sigma(1)$ the set of dimension one cones of $\Sigma$. 
Let $\rho_1, \dots, \rho_m$ be an enumeration of $\Sigma(1)$, and denote by $v_k$ the generator of $\rho_k \cap N$ ($1 \leq k \leq m $).
Assume furthermore that $\Sigma$ contains a cone of maximal dimension.
Without loss of generality, we can assume that the cone $\sigma$ generated by the elements $v_1, \dots, v_n$ belongs to $\Sigma$.
Therefore, the affine toric variety $Y$ determined by $\sigma$, which is an affine space of dimension $n$, is naturally an open subset of $Z$.

\bigskip

Let $\CC[x_{\rho_1}, \dots ,x_{\rho_m}]$ be  a polynomial ring in the independent variables $x_{\rho_k}$ ($1 \leq k \leq m$). 
By a well known result of Cox \cite{cox1995homogeneous}, any choice of a set of invariant divisors of $Z$ freely generating $\Pic(Z)$  determines an isomorphism

\begin{equation}
    \label{eq: cox toric}
    \Cox(Z) \simeq \CC[x_{\rho_1}, \dots ,x_{\rho_m}].
\end{equation}
The previous isomorphism is described by \cite[Proposition 1.1]{cox1995homogeneous}.
By Lemma \ref{lem:pic(Z) with nice open}, the divisors corresponding to the irreducible components of $Z \setminus Y$ freely generate $\Pic(Z)$.
From now on, we fix the isomorphism between the rings $\Cox(Z)$ and $\CC[x_{\rho_1}, \dots ,x_{\rho_m}]$ determined by the previously mentioned set of divisors and, with little abuse of notation, we write $\Cox(Z)= \CC[x_{\rho_1}, \dots ,x_{\rho_m}].$
We claim that the variables $x_{\rho_i}$ are the cluster variables of the minimal monomial lifting of the trivial cluster structure on $Y$.

\bigskip

Let $M:= \Hom_\ZZ(N,\ZZ)$ be the dual lattice of $N$, and let $(v_1^*, \dots, v_n^*)$ be the dual base of $(v_1, \dots , v_n)$. 
We denote by $\langle -, - \rangle: M \times N \longto \ZZ$ the natural pairing.
The lattice $M$ is canonically identified with the character group of the torus acting on $Z$. 
The character associated to an element $m \in M$ is denoted by $\chi^m.$
Then, we have that 
$$
\Orb_Y(Y) = \CC[ \chi^{v_1^*}, \dots , \chi^{v_n^*}].
$$
In particular, $\Orb_Y(Y)$ is the upper cluster algebra of the seed $t$ with $n$ highly frozen variables $\chi^{v_1^*}, \dots , \chi^{v_n^*}$, and no other variable.
For $1 \leq k \leq m$, we denote by $E_k$ the invariant prime divisor of $Z$ determined by $\rho_k$, as considered for example in \cite[Section 3.3]{fultonToric}.
Then, we have that 
$$
Z \setminus Y= \bigcup_{k=n+1}^n E_k.
$$
Let 
$ D:= \{ n+1, \dots, m\}$, $L:= \Orb(E_d)_{d \in D}$ and, for $d \in D$, let $\sigma_d$ be the rational function $1$ seen as a global section of $L_d$. 
The triple $(\lX_L, L, (\sigma_d)_{d \in D})$ is the standard homogeneously suitable for lifting structure of $\Cox(Z)$ corresponding to $Y$.

Let $\lif t^D$ be the minimal monomial lifting of $t$ with respect to $(\lX_L, L, (\sigma_d)_{d \in D})$. 
Then, the upper cluster algebra $\uclu(\lif t^D)$ is a polynomial ring in the variables $\lif \chi^{v_1^*}, \dots, \lif \chi^{v_n^*}, \sigma_{n+1}, \dots , \sigma_m$. 
 For $d \in D$, because of  \cite[Proposition 1.1]{cox1995homogeneous} (and its proof), we have that the element  $x_{\rho_d} $ is the rational function 1 seen as a global section of $\Orb(E_d)$.
Therefore, we clearly have that $x_{\rho_d}=\sigma_d$.
Moreover, for $1 \leq k \leq n$, the element $x_{\rho_k}$ belongs to
\begin{center}
$
 \HH^0\bigl(Z, \, \Orb\bigl( \sum_{d \in D} -\langle v_k^*, v_d \rangle E_d \bigr)\bigr)
$
\end{center}
and corresponds to the rational function $\chi^{v_k^*}$.
In particular, recalling the definition of the map $\iota^*: \Cox(Z)= \Orb_{\lX_L}(\lX_L) \longto \Orb_Y(Y)$ (see the notation of Section \ref{sec: suitable lift schemes}), we have that 
$$
\iota^*(x_{\rho_k})= \chi^{v_k^*}.
$$
But, a standard calculation implies that 
$$
\val_{E_d}\big( \chi^{v_k^*} \bigr)= \langle v_k^*, v_d \rangle \qquad(d \in D, \, 1 \leq k \leq n).
$$
Therefore, the previous discussion together with Lemma \ref{lem: pole filtration} implies that 
$$
\lif \chi^{v_k^*}= x_{\rho_k} \qquad (1 \leq k \leq n).
$$
This completes the proof of the claim. In particular, it follows by \cite{cox1995homogeneous}, that
$$
\uclu(\lif t^D)= \Cox(Z).
$$

\subsubsection{A Fano surface}
This example is inspired by \cite[Example 8.31]{gross2018canonical}. 
I kindly thank Alfredo N\'ajera Ch\'avez for bringing this example to my attention.

Let $S$ be the graded algebra generated in degree one, over $\CC$, by the variables $\theta_0, \theta_1 , \dots, \theta_5$ satisfying the following relations 
$$\begin{array}{c c c }
  \theta_1\theta_3=\theta_2\theta_0 + \theta_0^2, & 
   \theta_2\theta_4=\theta_3\theta_0 + \theta_0^2,& 
    \theta_3\theta_5=\theta_4\theta_0 + \theta_0^2,\\
    \theta_4\theta_1=\theta_5\theta_0 + \theta_0^2, & 
    \theta_5\theta_2=\theta_1\theta_0 + \theta_0^2. &
\end{array}$$
Let $Z:=\Proj(S)$.
This is a smooth Fano surface endowed with a closed embedding $Z \longto \PP^5$ defined by the elements $\theta_i$.
Moreover, let $Y:=Z_{\theta_0}$ be the non-vanishing locus of $\theta_0$.
The open subset $Y$ of $Z$ is a smooth affine variety. 
Setting 
$$
x_i:= \frac{\theta_i}{\theta_0} \qquad (1 \leq i \leq 5),
$$
we have that the algebra $\Orb(Y)$ is presented, over $\CC$, by the elements $x_i$ satisfying the following set of relations
$$\begin{array}{c c c }
  x_1x_3=x_2+1 & 
   x_2x_4=x_3+1 &  
    x_3x_5=x_4+1 \\
     x_4x_1=x_5+1 & 
      x_5x_2=x_1+1. &
\end{array}$$
The ring $\Orb(Y)$ is an upper cluster algebra of type $A_2$. An initial seed $t$ is graphically represented by the quiver
\[\begin{tikzcd}
	{1 \bigcirc} && {2 \bigcirc}
	\arrow[from=1-1, to=1-3]
\end{tikzcd}\]
and has $x_1$ and $x_2$ as initial cluster variables. 
The set of cluster variables of the upper cluster algebra $\uclu(t)$ is finite, and they are the functions $x_1, \dots , x_5.$
Moreover, the ring $\Orb(Y)$ is factorial. 
An immediate way for showing this is using   \cite[Theorem 4.9]{cao2022valuation}.
Moreover, we have that $\Orb(Y)^\times=\CC^\times$ because of Theorem \ref{thm: inv upper}.
Therefore the pair $(Y, Z)$ satisfies the assumptions of Section \ref{sec: case of Cox}.  
We have that $Z \setminus Y= \bigcup_{k=1}^5 E_k$, where the five prime divisors $E_i$ are
$$
E_1:= \Proj\big(S / (\theta_1, \theta_2, \theta_3)\big) \quad E_2:=\Proj\big(S / (\theta_2, \theta_3, \theta_4)\big) \quad \dots \quad  E_5:=\Proj\big(S / (\theta_5, \theta_1, \theta_2)\big).
$$
A simple local computation allows to prove that the minimal lifting matrix $\nu$ is, with obvious notation:

$$\nu= \bmat 0 & -1\\
1 & 0\\
1 & 1\\
0 & 1\\
-1 & 0
\emat, \quad \text{and therefore} \quad 
-\nu B= \bmat -1 & 0\\
0 & -1\\
1 & -1\\
1 & 0\\
0 & 1
\emat. 
$$
In particular, the seed $\lif t^D$ is represented by the following quiver 
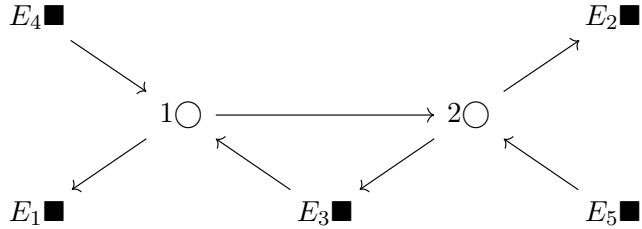
\[\begin{tikzcd}
	{E_4 \blacksquare} &&&& {E_2 \blacksquare} \\
	& {1 \bigcirc} && {2 \bigcirc} \\
	{E_1 \blacksquare} && {E_3 \blacksquare} && {E_5 \blacksquare}
	\arrow[from=2-2, to=2-4]
	\arrow[from=1-1, to=2-2]
	\arrow[from=2-2, to=3-1]
	\arrow[from=2-4, to=1-5]
	\arrow[from=3-5, to=2-4]
	\arrow[from=2-4, to=3-3]
	\arrow[from=3-3, to=2-2]
\end{tikzcd}\]
and has initial mutable cluster variables 
$$
\lif x_1= x_1 \frac{\sigma_2 \sigma_3}{\sigma_5} \quad \lif x_2= x_2 \frac{\sigma_3 \sigma_4}{\sigma_1}
$$
where the $\sigma_i$ are the sections of definition of the standard  homogeneously suitable for lifting structure of $\Cox(Z)$ corresponding to the open subset $Y$.
Note that $\uclu(\lif t^D)$ is the upper cluster algebra of type $A_2$ with universal coefficients.

The cluster algebra $\clu(\lif t^D)$ has $5$ mutable cluster variables, that we denote, with obvious notation, by $\lif x_i$ ($1 \leq i \leq 5$). 
Moreover, we have that
\begin{center}
$
\lif x_i \in \HH^0\bigl( Z , \, \Orb\bigl(E_{i+1} + E_{i+2}- E_{i+4} \bigr) \bigr) \qquad (1 \leq i \leq 5),
$
\end{center}
where in the previous formula, the indices of the $E_j$ have to be intended modulo $5.$

\section{Partial flag varieties}
\label{sec: flag varieties}

We need to introduce some notation and classical results on reductive groups.
Let $G$ be a complex semisimple simply connected algebraic group, $B$ and $B^-$ be opposite Borel subgroups, $U$ and $U^-$ be the unipotent radicals of $B$ and $B^-$ respectively. 
Consider the maximal torus of $G$
$$
H:= B \cap B^-.
$$ 
We denote by $\Phi$ the root system of $G$, with its set of simple roots $\Delta$ determined by $B$.
The notation $\Phi^+$ (resp. $\Phi^-$) stands for the set of positive (resp. negative) roots.
We denote by $\varpi_\alpha \, (\alpha \in \Delta)$ the fundamental weights.
Let $W$ be the Weyl group of $G$, with its Coxeter generators $s_\alpha \, (\alpha \in \Delta)$.
We denote by $w_0$ the longest element of $W$.
For an element $w \in W$, we consider the unipotent algebraic groups 
\begin{equation}
    \label{eq: U(w)}
    U(w):= U \cap w\inv U^- w ,\qquad U^-(w):= U^- \cap w\inv U w.
\end{equation}

From now on, we fix a set $K \subsetneq \Delta$ and we define $J:= \Delta \setminus K$. 
Let $P$ (resp. $P^-$) be the parabolic subgroup of $G$ generated by $B$ (resp. $B^-$) and by the root groups $U(s_\alpha)$ (resp. $U^-({s_\alpha})$) for $\alpha \in K$. 
 Let 
 $$
 Z:= G/P \qquad Z^-:= P^- \setminus G
 $$
 be the \textit{partial flag varieties} of $G$ associated to $P$ and $P^-$. 
 These are smooth projective varieties, which are homogeneous spaces under the natural action of $G$ by left and right multiplication respectively.

\bigskip

Let 
$$
\Phi_K:= \Phi \cap \bigoplus_{\alpha \in K} \ZZ \alpha \quad \text{and} \quad W_K:= \langle s_k \, : \, k \in K\rangle \subseteq W.
$$  
We denote by $w_{0,K}$  the longest element of $W_K$ and we consider (pay attention the minus signs exponents)
$$
z=z_K:= w_0 w_{0,K}, \qquad Y:= U^-(Z), \qquad Y^-:=U(z).
$$ 
The quotient maps $G \longto Z$ and $G\longto Z^-$ identify $Y$ with the open $B^-$-orbit of $Z$ and $Y^-$ with the open $B$-orbit of $Z^-$ respectively.
As $Y$ and $Y^-$ are affine spaces, the pairs $(Y,Z)$ and $(Y^-,Z^-)$ satisfy the assumptions of Section \ref{sec: case of Cox}. 
Therefore, Lemma \ref{lem: min mon cox affine} implies that we can lift any cluster structure of maximal rank on the rings $\Orb_Y(Y)$ and $\Orb_{Y^-}(Y^-)$ inside the rings $\Cox(Z)$ and $\Cox(Z^-)$ respectively.

\bigskip

In order to compare the cluster algebras obtained by minimal monomial lifting with the ones constructed in \cite[Section 10]{geiss2008partial}, we need to take a closer look at this situation.
We start by describing $\Pic(Z)$ and the standard homogeneously suitable for lifting structure of $\Cox(Z)$ corresponding to the open subset $Y$.

\bigskip

It is well known that the restriction map $X(P) \longto X(H)$ is injective and its image is the lattice $\bigoplus_{j \in J} \ZZ \varpi_j.$ 
Considering the fundamental weights $\varpi_j \, (j \in J)$ as  characters of $P$, allows to introduce the subgroup $P_1$ of $P$ defined by
$$
P_1:= \bigcap_{j \in J} \ker( \varpi_j ).
$$
  Then, the group
 $$
 T:=P/P_1
 $$
 is a torus, and the natural projection $\pi: P \longto T$ induces an isomorphism between the groups $X(T)$ and $X(P)$.
 From now on, we freely identify the groups $X(T), X(P)$ and  $\bigoplus_{j \in J} \ZZ \varpi_j$ as described above. 
 Let 
 $$
 \lX:= G/P_1,
 $$
 and $p:\lX \longto Z$ be the natural projection. Since $P$ normalizes $P_1$, 
  the torus $T$ naturally acts on $\lX$ by means of the following formula
  $$
  t \cdot (g P_1/P_1):= gt P_1/P_1 \qquad (g \in G, \, t \in T).
  $$
  Moreover, the map $p$ is a Zarisky locally trivial $T$-bundle. 
  The following lemma should be well known to specialists (see for example \cite[Section 4]{arzhantsev2009factoriality}), we add a proof for completeness and to recall some important aspects of it.

\begin{lemma}
    \label{lem:cox GP} The following hold.
    \begin{enumerate}
        \item We have a canonical isomorphism $  \Pic(Z) \simeq X(T)=X(P)$.
        \item The $T$-bundle $p: \lX \longto Z$ is the canonical morphism of the relative spectrum of the Cox sheaf of $Z$.
    \end{enumerate}
\end{lemma}

\begin{proof}
    For $\lambda \in X(P)$, we consider the one dimensional $P$-representation  $\CC_{-\lambda}$ on which $P$ acts via the character $-\lambda$. Consider the $P$-action on $G \times \CC_{-\lambda}$ defined by 
    $$
    p\cdot (g,v)=\big(g p\inv, (-\lambda)(p)v\big) \qquad (p \in P, \, g \in G, \, v \in \CC_{-\lambda}).  
    $$
    It is a standard fact that the geometric quotient for this action exists and it is denoted by 
    \begin{equation}
    \label{eq: L lambda}
          G \times_P \CC_{-\lambda}=:\Li_\lambda.
    \end{equation} 
    The class of an element $(g,v) \in G \times \CC_{-\lambda}$ is denoted by $[g:v] \in \Li_\lambda$.
    The map $\Li_\lambda \longto Z$ defined by $[g : v] \longto gP/P$ turns $\Li_\lambda$ into a line bundle on $Z$.
    As $G$ is simply connected, it is then a classical fact that this construction gives an isomorphim $X(P) \simeq \Pic(Z)$.
    See for example \cite[Section 4]{arzhantsev2009factoriality}.

    \bigskip
    
   Note that, for $\lambda \in X(P)$, the line bundle $p^*\Li_\lambda$ is trivial.
   Indeed, since the character $\lambda$ is trivial on $P_1$, we have a morphism $\lX \longto \Li_\lambda $ of schemes over $Z$ defined by 
   $$
   gP_1/P_1 \longto [g : 1_{-\lambda}],
   $$ 
   where $1_{-\lambda}$ is the element $1$ of the representation $\CC_{-\lambda}$. 
   The induced morphism $\lX \longto p^* \Li_\lambda$
   is a never vanishing global section of $p^*\Li_\lambda$.
  Moreover, we have an isomorphism of $\Orb_Z$-modules 
   \begin{equation}
       \label{eq:Cox GP 0}
       \Li_\lambda \simeq \bigl(p_* \Orb_{\lX}\bigr)^{(T)}_\lambda
   \end{equation}
  
which, for any open subset $U$ of $Z$, sends a section $\sigma \in \Li_\lambda(U)$ to the semi-invariant function $f_\sigma \in \Orb_{\lX}(p\inv(U))$ defined by
$$
\sigma(gP/P)= [g : f_\sigma(gP_1/P_1) 1_{-\lambda}]  \qquad \big( (gP_1/P_1 \in p\inv(U) \big).
$$
   Note that, for any $\lambda, \mu \in X(P)$, we have an isomorphism $\Li_{\lambda + \mu} \simeq \Li_\lambda \otimes \Li_\mu$ induced by the isomorphism of $P$-modules $\CC_{-\lambda} \otimes \CC_{- \mu} \simeq \CC_{-(\lambda +\mu)}$ sending $1_{-\lambda}\otimes 1_{-\mu}$ to $1_{-(\lambda +\mu)}$. 
  In particular, the sheaf 
  $$
   \Rc := \bigoplus_{\lambda \in X(P)} \Li_\lambda
  $$
  has a natural structure of $X(P)= \Pic(Z)$-graded $\Orb_Z$-algebra, and we have an isomorphism of $\Pic(Z)$-graded $\Orb_Z$-algebras
   \begin{equation}
       \label{eq: R pistar GP}
        \Rc \simeq p_* \Orb_{\lX}.
   \end{equation}
  Isomorphism \eqref{eq: R pistar GP} is defined on the homogeneous components by \eqref{eq:Cox GP 0}.
By the first statement of the lemma, $\Rc$ is the Cox sheaf of $Z$.
  Since the morphism $p: \lX \longto Z$ is affine, Isomorphism \eqref{eq: R pistar GP} corresponds to an isomorphism $\lX \simeq \Specbf(\Rc)$ of schemes over $Z$. 
  This follows in fact from the universal property of the relative spectrum of a quasi-coherent sheaf of $\Orb_Z$-algebras.
\end{proof}
 For  $j \in J$, we denote by $E_j:=\overline{ B^- s_j P/P}$.
We have that the $E_j \, (j \in J)$ are pairwise distinct irreducible closed subsets of $Z$ and moreover 
$$
Z \setminus Y = \bigcup_{j \in J} E_j.
$$ 

Recall that to any simple root $\alpha$, Fomin and Zelevinsky \cite{fomin1999double} associate a function $\Delta_{\varpi_\alpha,\varpi_\alpha} \in \CC[G]$, called \textit{generalized minor}.
For any $j \in J$, the function $\Delta_{\varpi_j,\varpi_j}$ is actually $P_1$-right invariant.
 Therefore, $\Delta_{\varpi_j,\varpi_j}$ descends to a function on $\lX$, that we denote by $\sigma_j$.

\begin{coro}
    \label{coro:standard lifting cox GP}
   Let $\Li:= (\Li_{\varpi_j})_{j \in J} $ and $\sigma:=(\sigma_j)_{j \in J}$. The triple $(\lX, \Li, \sigma)$ is the standard homogeneously suitable for lifting structure of $\Cox(G/P)$ corresponding to the open subset $Y$.
\end{coro}

\begin{proof}
   The isomorphism $\lX \simeq \Specbf (\Rc_\Li)$ is established in the proof of Lemma \ref{lem:cox GP}.
  The function $\sigma_j \, (j \in J) $ is semi-invariant for the $T$-action and its weight is $\varpi_j$. 
   This elementary fact can be found for example in  \cite{fomin1999double}.
   Hence, we can interpret $\sigma_j$ has a global section of the line bundle $\Li_{\varpi_j}$.
   Moreover, we have that $\divv(\sigma_j)=E_j$ by 
  \cite[Lemmas 5.2.1, 5.2.3]{francone2023minimal}.
  In particular, under the isomorphism 
$$ 
\begin{array}{rcl}
   \Li_{\varpi_j} & \longto & \Orb(E_j)\\[0.3em]
   \tau & \longmapsto & \displaystyle \frac{\tau}{\sigma_j},
   \end{array}
   $$
   the section $\sigma_j$ corresponds to the rational function $1$, on $Z$, seen as a global section of $\Orb(E_j)$.
   Since $Z \setminus Y= \displaystyle \bigcup_{j \in J} E_j$, the statement follows.
\end{proof}

For a dominant character $\lambda$ of $H$, we denote by $V(\lambda)$ the irreducible $G$-representation of highest weight $\lambda$, and by $V(\lambda)^*$ its dual.
We extend this notation by saying that, for a character $\lambda$ of $H$, then $V(\lambda)= \{0\}= V(\lambda)^*$ is $\lambda$ is not dominant.
 Note that $\Orb_\lX(\lX)$ has a natural structure of $G$-module induced by the formula
$$
(g\cdot f)(x)= f(g\inv x)  \qquad (f \in \Orb_\lX(\lX), \, g \in G,  \, x \in \lX).
$$ 
As a $G$-module, the space $\Orb_\lX(\lX)$ is multiplicity free and its decomposition into irreducible $G$ representations agrees with its decomposition into spaces of semi-invariant functions for the $T$-action.
More precisely, we have that: 

\begin{equation}
    \label{eq:Cox GP 1}
    \begin{array}{l}
\Orb_\lX(\lX) =  \displaystyle \bigoplus_{\lambda \in X(T)}\Orb_\lX(\lX)_\lambda \qquad \text{and} \qquad \Orb_\lX(\lX)_\lambda \simeq V(\lambda)^*.
\end{array}
\end{equation}
Recalling  the Isomorphisms \eqref{eq:Cox GP 0}, the decomposition \eqref{eq:Cox GP 1} is exactly the Borel-Weil theorem.

\bigskip

For  $j \in J$, let $v_j^+ \in V(\varpi_j)$ be a highest weight vector. 
Let $\iota$ be the $G$-equivariant closed embedding defined by
\begin{equation*}
    \begin{array}{rcl}
        \iota : Z & \longto &  \prod_{j \in J}\PP\big(V(\varpi_j)\big)\\[1em]
        gP/P & \longmapsto & \displaystyle \big([g v_j^+]\big)_{j \in J}.
    \end{array}
\end{equation*} 
Similarly, we have a $G$-equivariant closed embedding 
\begin{equation*}
    \begin{array}{rcl}
       \lX & \longto &   \prod_{j \in J} \bigl(V(\varpi_j) \setminus \{0\} \bigr)\\[1em]
       gP_1/P_1 & \longmapsto & \big(gv_j^+\big)_{j \in J}.
    \end{array}
\end{equation*}
 These maps fit into a cartesian diagram:

\begin{equation}
    \label{eq: Cox GP cartesian}
    \begin{tikzcd}
	\lX && {\prod_{j \in J} \bigl(V(\varpi_j) \setminus \{0\} \bigr)} \\
	Z && {\prod_{j \in J}\PP\big(V(\varpi_j)\big)}
	\arrow["\iota", from=2-1, to=2-3]
	\arrow[from=1-1, to=1-3]
	\arrow[from=1-3, to=2-3]
	\arrow["p"', from=1-1, to=2-1]
\end{tikzcd}
\end{equation}

Let $\overline{ \lX} \subseteq \prod_{j \in J} V(\varpi_j)$ be the closure of $\lX$.
We consider $\overline{\lX}$ as a closed subscheme of $\prod_{j \in J} V(\varpi_j)$ with the reduced structure. 
The affine variety $\overline{ \lX}$ has a natural structure of $G$-variety, inherited from the diagonal $G$-action on the space $\prod_j V(\varpi_j)$.
The ring 
$$
\CC[Z]:=\Orb_{\overline{ \lX}}(\overline{ \lX})
$$
is known in the literature as the \textit{multi-homogeneous} coordinate ring of $Z$ coming from the closed embedding $\iota$. 
The scheme $\overline{ \lX}$ is just the affinisation of $\lX$ because of the following lemma, which should be well known to specialists, but for which we add a proof by lack of a precise reference.

\begin{lemma}
    \label{lem: Cox GP and multihom coord ring}
   The restriction map 
    $
   \CC[Z] \longto \Orb_\lX(\lX)
    $
    is an isomorphism.
\end{lemma}

\begin{proof}
Note that $\lX$ is an open subset of $\overline{\lX}$.
Indeed, this is  a topological consequence of the fact that the map $\lX \longto {\prod_{j \in J} \bigl(V(\varpi_j) \setminus \{0\} \bigr)}$ is a closed embedding and that the space $ {\prod_{j \in J} \bigl(V(\varpi_j) \setminus \{0\} \bigr)} $ is open in $\prod_{j \in J} V(\varpi_j) $. 
Otherwise, one can observe that $\lX$ a single $G$-orbit, and it is therefore open in its closure.
 Recall that $\overline{\lX}$ is, by definition, reduced.
 Then, as $\overline{\lX}$ is irreducible, being the closure of an irreducible space, it follows that the restriction map $\CC[Z] \longto \Orb_\lX(\lX)$ is injective.
 Since $\Orb_\lX(\lX)$ is generated, as a $\CC$-algebra, by the components $V(\varpi_j)^* \, (j \in J)$ of Eq. \eqref{eq:Cox GP 1}, we just need to prove that the $G$-representation $\CC[Z]$ contains a component isomorphic to $V(\varpi_j)^*$, for any $j \in J.$
 
  Note that we have a surjective restriction map
   \begin{equation}
       \label{eq: multi cohord GP}
       \Sym \biggl( \bigoplus_{j \in J} V(\varpi_j)^* \biggr) \longto \CC[Z],
   \end{equation}
   induced by the closed immersion $\overline{\lX} \subseteq \prod_{i \in J} V(\varpi_j)$.
   Here, $\Sym$ denotes the symmetric algebra in its argument.
  Moreover, for any $j \in J$, Eq. \eqref{eq: multi cohord GP} restricts to an injective map of $G$-modules $V(\varpi_j)^* \longto \CC[Z]$.
  The statement follows.
  \end{proof}

We denote by $(-)^T:G \longto G$ the \textit{transpose} map. 
This is an anti-authomorphism of algebraic groups, defined for example in \cite[Formula (2.1)]{fomin1999double}. 
 Let 
 $$
 P^-_1:=P_1^T, \qquad \lX^- := P^-_1 \setminus G.
 $$ 
 The space $\lX^-$ carries a natural action of the torus $T^-:= P^-_1 \setminus P^-$.
 The transpose map induces isomorphisms of algebraic varieties that fit into the following commutative diagram
  \[\begin{tikzcd}
	{Y^-} & {Z^-} & {\lX^-} & T^- \\
	{Y} & Z & \lX & T.
	\arrow["{(-)^T}", from=1-1, to=2-1]
	\arrow[from=1-1, to=1-2]
	\arrow["{(-)^T}", from=1-2, to=2-2]
	\arrow[from=2-1, to=2-2]
	\arrow[from=1-3, to=1-2]
	\arrow[from=2-3, to=2-2]
	\arrow["{(-)^T}", from=1-3, to=2-3]
 \arrow["{(-)^T}", from=1-4, to=2-4]
\end{tikzcd}\]
This simple observation easily allows to use the previous results for the pair $(Y,Z)$ to obtain analogue results for the pair $(Y^-, Z^-)$.
In particular, for $j \in J$, let $\Li_j^-$ (resp. $\sigma_j^-$) the pullback of the line bundle $\Li_j$ (resp. the function $\sigma_j$) along the isomorphism $(-)^T : Z^- \longto Z$ (resp. $\lX^- \longto \lX$). 
We use the notation $\Li^-:= (\Li_j^-)_{j \in J}$ and $ \sigma^-:= (\sigma^-_j)_{j \in J}$.
Then, Corollary \ref{coro:standard lifting cox GP} implies that 
$
(\lX^-, \Li^-, \sigma^-)
$
is the standard homogeneously suitable for lifting structure of $\Cox(Z^-)$ with respect to $Y^-$.

Note that the ring $\Orb_{\lX^-}(\lX^-)$ has a natural structure of $G$-module defined by the formula
$$
(g \cdot f)(x):= f(x g) \qquad (g \in G, \, f \in \Orb_{\lX^-}(\lX^-), \, x \in \lX^-).
$$
Rephrasing \eqref{eq:Cox GP 1}, we have that the spaces of $T^-$ semi-invariant functions on $\lX^-$ are $G$-stable. 
Moreover:

\begin{equation}
    \label{eq:Cox P-G 1}
    \begin{array}{l }
\Orb_{\lX^-}(\lX^-) =  \displaystyle \bigoplus_{\lambda \in X(T^-)} \Orb_{\lX^-}(\lX^-)_\lambda \qquad   \text{and} \qquad \Orb_{\lX^-}(\lX^-)_\lambda  \simeq V(\lambda).
\end{array}
\end{equation}

  Finally, if we denote by $\CC[Z^-]$ the multi-homogeneous coordinate ring of $Z^-$ corresponding to the embedding
$$
Z^- \xlongrightarrow{(-)^T} Z \xlongrightarrow{\iota} {\prod_{j \in J}\PP(V(\varpi_j))},
$$
we have a natural isomorphism, induced by restriction

\begin{equation}
    \label{eq: Cox GP and multihomogeneous}
    \CC[Z^-] \simeq \Orb_{\lX^-}(\lX^-)= \Cox(Z^-).
\end{equation}
We are now in position to discuss cluster structures on the ring $\Cox(Z^-)$.

\bigskip

In \cite{goodearl2021integral}, the authors identify the ring $\Orb_{Y^-}(Y^-)$ with an explicit upper cluster algebra that is also equal to the corresponding cluster algebra. 
Certain seeds of this cluster structure are defined in terms of  reduced expressions  of the Weyl group element $w_0w_{0,K}$. 
The seed corresponding to the reduced expression $\ii$ is denoted by $t_\ii$, and is described in \cite[Section 7]{goodearl2021integral}.
When $G$ is simple and simply laced, this cluster structure on the ring $\Orb_{Y^-}(Y^-)$ had already been constructed, by different methods, in \cite{geiss2011kac}.

By the discussion at the beginning of this section, if $\ii$ is a reduced expression of $w_0w_{0,K}$ and $\lif t_\ii^J$ is the minimal monomial lifting of the seed $t_\ii$ with respect to the standard homogeneously suitable for lifting structure of $\Cox(Z^-)$ corresponding to $Y^-$, we have an inclusion of $\Pic(Z^-)$-graded algebras 
$$
\uclu(\lif t_\ii^J) \subseteq \Cox(Z^-).
$$

\begin{theo}
\label{thm: complete flag}
    Assume that $K= \emptyset$, so that $Z^-= B^- \setminus G$ is the complete flag variety. 
    For any reduced expression $\ii $ of $w_0$, the upper cluster algebra
    $\uclu(\lif t_\ii^J)$ coincides with the ring $\Cox(Z^-)$.

\end{theo}

\begin{proof}
   By a well known result of \cite{vinberg1972class}, we have that $\lX^-= U^- \setminus G$ is an open subset of its affinization whose complement is of codimension strictly greater than one.
    Therefore, the statement is just a reformulation of \cite[Theorem 8.3.2]{francone2023minimal}. 
\end{proof}

Assume  now  that $G$ is simple and simply laced. 
In \cite{geiss2008partial},  the authors construct a cluster algebra $\clu_J$ as a subalgebra of the ring $\Orb_{Y^-}(Y^-)$. 

\begin{remark}
    This  cluster structure is related  to the one considered in \cite{goodearl2021integral} and \cite{geiss2011kac}, that we were discussing before Theorem \ref{thm: complete flag}, but is different.
We refer to \cite[Remark 17.5]{geiss2011kac} for more details.
\end{remark}

By \cite[Theorem 17.4]{geiss2011kac}, we have that $\clu_J=\Orb_{Y^-}(Y^-)$.
Moreover, the starfish lemma (see for example \cite{fomin2021introduction6} or  \cite[Lemma 4.0.6]{francone2023minimal}) and \cite[Theorem 1.3]{geiss2013factorial} imply that we have equalities
$$
\clu_J=\uclu_J= \Orb_{Y^-}(Y^-),
$$
where $\uclu_J$ is the upper cluster algebra of $\clu_J$.
In \cite[Section 10]{geiss2008partial}, the cluster algebra $\clu_J$ is used to construct a cluster algebra $\widetilde \clu_J$ as a subring of $\CC[Z^-]$.
 The algebra $\widetilde \clu_J$ is obtained by an homogenization procedure \cite[Sections 10.1, 10.2, Theorem 10.2]{geiss2008partial} which, because of the following theorem, is a special issue of minimal monomial lifting.
   
\begin{theo}
    \label{thm:Cox GP geiss e minimal mon lifting}
    Let $t$ be a seed of $\clu_J$, and  $\widetilde t$ be the associated seed of $\widetilde \clu_J$ defined in \cite[Section 10.2]{geiss2008partial}.
    Moreover, let $\lif t^J$ be the minimal monomial lifting of the seed $t$ with respect to the standard homogeneously suitable for lifting structure of $\Cox(Z^-)$ corresponding to the open subset $Y^-$. 
    Through Isomorphism \eqref{eq: Cox GP and multihomogeneous}, we have that 
    $$
    \widetilde t= \lif t^J. 
    $$
\end{theo}

\begin{proof}
 Let
  $$
  t= (I_{uf}, I_{sf}, I_{hf}, B,x) \quad \text{and} \quad \widetilde t= (\widetilde I_{uf}, \widetilde I_{sf}, \widetilde I_{hf} , \widetilde B, \widetilde x).
  $$
  By the definition of $\widetilde t$, we have that
  $$
  \widetilde I_{uf}=I_{uf}, \quad \widetilde I_{sf}=I_{sf}= \emptyset, \quad \widetilde I_{hf}=I_{hf} \cup J, \quad \widetilde B_{|I \times I_{uf}}=B.
  $$
In particular, the seeds $\widetilde t$ and $\lif t^J$ have the same sets of vertices.
By the definition of $\widetilde t$, the frozen variable $\widetilde x_j \, (j \in J)$ is the generalized minor $\Delta_{\varpi_j, \varpi_j} \in \Orb_{\lX^-}(\lX^-)$. 
  But, by \cite[Proposition 2.7]{fomin1999double}, we have that 
  $$
  \sigma^-_j= \Delta_{\varpi_j, \varpi_j} \qquad(j \in J).
  $$ 
  Then, let $i \in I$. 
  The cluster variable $\widetilde x_i$ is, by definition, the homogenization of $x_i$ as defined in Corollary \ref{cor:cox, to compare Leclerc}.
  Therefore, we have that 
  $$
  \lif x_i = \widetilde x_i.
  $$
  Hence, the seeds $\widetilde t$ and $\lif t^J$ have the same cluster variables.
  Let $\nu \in \ZZ^{J \times I }$ be the minimal lifting matrix of the seed $t$ with respect to the standard homogeneously suitable for lifting structure of $\Cox(Z^-)$ corresponding to the open subset $Y^-$.
  In order to conclude the proof, it only remains to show that 
  \begin{equation}
      \label{eq: aux 1}
      \widetilde B_{| J \times I_{uf}}= - \nu B.
  \end{equation} 
 By \cite[Theorem 10.2]{geiss2008partial}, we have that for any $k \in I_{uf}$, the formula $\mu_k(\widetilde x)_k= \widetilde{\mu_k(x)_k}$ holds. 
 Therefore, Corollary \ref{cor:cox, to compare Leclerc} and the previous discussion shows that the assumptions of \cite[Proposition 3.0.9]{francone2023minimal} are satisfied. 
 Thus, \cite[Proposition 3.0.9]{francone2023minimal} implies that \eqref{eq: aux 1} holds.
 \end{proof}

Let $t$ be a seed of the cluster algebra $\clu_J$.
Lemma \ref{lem: min mon cox affine} implies the following improvement of Statement (i) of \cite[Theorem 10.2]{geiss2008partial}.

\begin{coro}
    The upper cluster algebra $\uclu(\lif t^J)$ is a $\Pic(Z^-)$-graded subalgebra of $\Cox(Z^-).$
\end{coro}

Let $\nu \in \ZZ^{J \times I }$ be the minimal lifting matrix of the seed $t$ with respect to the standard homogeneously suitable for lifting structure of $\Cox(Z^-)$ corresponding to the open subset $Y^-$.
Moreover, let $k$ be a mutable vertex of $t$ and $t':= \mu_k(t)$.
Let $\nu':= \mu_k(\nu) \in \ZZ^{J \times I}$ be the mutation of the lifting matrix $\nu$, as defined in \cite[Definition 3.0.6]{francone2023minimal}.

\begin{coro}
    \label{coro:mutation min lift matrix GP}
  The matrix $\nu'$ is the minimal lifting matrix of the seed $t'$ with respect to the standard homogeneously suitable for lifting structure of $\Cox(Z^-)$ corresponding to the open subset $Y^-$.
\end{coro}

\begin{proof}
   This is an immediate consequence of \cite[Lemma 3.0.8]{francone2023minimal}, \cite[Theorem 10.2]{geiss2008partial}  and Theorem \ref{thm:Cox GP geiss e minimal mon lifting}. 
\end{proof}

Corollary \ref{coro:mutation min lift matrix GP} means that the order of pole of the non-initial cluster variables of the algebra $\uclu_J$, along the boundary divisors of $Y^-$ in $Z^-$, can be computed by iteratively applying the formulas of \cite[Definition 3.0.6]{francone2023minimal} to the matrix $\nu$.
Recall that the matrix $\nu$ encodes the order of pole of the initial cluster variables of the algebra $\uclu_J$ along the boundary divisors of $Y^-$ in $Z^-$.

\begin{remark}
\label{rem: cox GP valuations preproj alg}
In the notation of Theorem \ref{thm:Cox GP geiss e minimal mon lifting}, it is interesting to compare the definition of the seed $\lif t^J$ to Geiss-Leclerc-Schr\"oer's definition of the seed $\widetilde t$.
One intriguing observation that follows from Theorem \ref{thm:Cox GP geiss e minimal mon lifting}, is that the order of pole of the cluster variables of the seed $t$, along the irreducible components of $Z^- \setminus Y^-$, can be expressed in terms of the dimension of the spaces of morphisms between certain modules over the preprojective algebra of the Dynkin diagram of $G$.
It would be interesting to further investigate this relation. 
\end{remark}

\section{The diagonal partial compactification of the finite cluster variety}
\label{sec: diag comp cluster manifold}

\textbf{Assumptions.} In this section, $t$ is a seed of maximal rank of an ambient field $\KK$. 
We assume that $t$ has no semi-frozen vertex, and that the upper cluster algebra $\uclu(t)$ is factorial.

\begin{remark}
    The factoriality of an upper cluster algebra  can be characterized, locally, thanks to \cite[Theorem 4.9]{cao2022valuation}.
    Examples of upper cluster algebras satisfying the previous assumptions are common in the literature, and can be found for example in \cite{goodearl2021integral}, \cite{geiss2011kac}, \cite{geiss2013factorial}, \cite{francone2023minimal}, \cite{galashin2023braid}.
\end{remark} 

As usual, we use the notation $t=(I_{uf}, I_{sf}= \emptyset, I_{hf}, B, x)$. 
For $i \in I_{uf} \cup \{ 0 \}$, we introduce the auxiliary notation
$$
t(i):= \begin{cases}
    t & \text{if} \, \, i=0\\
    \mu_i(t) & \text{otherwise}.
\end{cases} \quad \text{and} \quad  t(i)= \bigl(I_{uf}, I_{hf}, B(i), x(i) \bigr).
$$

Let $\Delta$ be the set of seeds mutation-equivalent to $t$. 
Since every frozen vertex of $t$ is highly frozen, the coeffincient ring $\coef[\Delta]$ of the cluster algebra of the seed $t$ is a polynomial ring. Namely, we have that: 
$$
\coef[\Delta]= \CC\bigl[x(0)_j\bigr]_{j \in I_{hf}}.
$$
Then, for $i \in I_{uf} \cup \{0\}$, we set 
$$
T(i):= \Spec\biggl(\Li(t(i))\biggr), \qquad  \overline{T(i)}:= \begin{cases}
T(i) & \text{if} \quad i=0\\
\Spec \biggl( \coef[\Delta][x(i)_i][x(i)^{\pm 1}_j]_{j \in I_{uf} \setminus \{i \}} \biggr) & \text{otherwise}
, \end{cases}
$$
where the notation $\Li\bigl(t(i) \bigr)$ is defined in Section \ref{sec: cluster upper cluster}.
Note that we have a canonical inclusion morphism $T(i) \longto \overline{T(i)}$. This map is an isomorphism when $i=0$, and for $i \in I_{uf}$, it identifies $T(i)$ with the principal open subset of $\overline{T(i)}$ defined by the non-vanishing locus of the function $x(i)_i$.

\bigskip

These two families of schemes can be glued along the maps defined by the mutation process. 
The details of the gleuing procedure are explained below.

\bigskip

For $i \in I_{uf}$, we denote  by
$$
M^+(i):= x(0)^{B(0)_{\bullet i}^+}= x(i)^{B(i)_{\bullet i}^-} \quad \quad M^-(i):= x(0)^{B(0)_{\bullet i}^-}= x(i)^{B(i)_{\bullet i}^+} 
$$
the two monomials in the exchange relation determined by the vertex $i$ of the seed $t(0)$.
Moreover, let $T(0)_{M^+(i)+M^-(i)}$ (resp. $T(i)_{M^+(i)+M^-(i)}$) be the open subset of the variety $T(0)$ (resp. $T(i)$) where the function $M^+(i)+M^-(i)$ does not vanish.

For $i,j \in I_{uf}\cup \{0\}$ with $i \neq j$, we define an open subset $U_{ij}$ (resp. $U_{ji}$)  of $T(j)$ (resp. $T(i)$) and an isomorphism $\mu_{ij} : U_{ij} \longto U_{ji} $ as follows.
\begin{enumerate}
    \item If $i \in I_{uf}$, then
    $$
    U_{i0}:= T(0)_{M^+(i)+M^-(i)}, \qquad U_{0i}:= T(i)_{M^+(i)+M^-(i)},
     $$
     and $\mu_{i0}$ is the isomorphism  whose pullback of regular functions is defined by the formula
\begin{equation}
    \label{eq:glueing clust variety}
\mu_{i0}^*\bigl(x(i)_j\bigr):= \begin{cases}
    x(0)_j & \text{if} \quad j \neq i\\[0.6em]
   \frac{M^+(i) + M^-(i)}{x(0)_i} & \text{otherwise}.
\end{cases}
\end{equation}
Moreover, we set $\mu_{0i}:= \mu_{i0}^{-1}$.
\item If $i, j \in I_{uf}$  and $i \neq j$, then
$$
U_{ji}:=\mu_{i0} \bigl( U_{i0} \cap U_{j0} \bigr), \quad \text{and} \quad \mu_{ji}:=  \mu_{j0} \circ \mu_{0i}.
$$
\end{enumerate}


One can easily verify that the set of functions $\mu_{ij} \, (i,j \in I_{uf} \cup \{0\}, \, i \neq j)$ is a well defined glueing data.

   \begin{definition}
       Let $Y$ (resp. $Z$) be the scheme obtained by glueing the $T(i)$ (resp. $\overline{T(i)}$ ) along the open subsets $U_{ji}$, via the maps $\mu_{ji}$.
       We refer to $Y$ as the \textit{finite cluster variety} of the seed $t$, and  to $Z$ as its \textit{diagonal partial compactification}.
   \end{definition}
   
Operationally, the scheme $Y$ (resp. $Z$)  is covered by open subsets $Y(i)$ (resp.  $Z(i)$), indexed by $i$ in $I_{uf} \cup \{0\} $.
Furthermore, for any index $i$, there are canonical isomorphims $\iota_i : T(i) \longto Y(i)$ (resp. $\overline{\iota_i}: \overline{T(i)} \longto Z(i)$) such that, for any $i,j \in I_{uf} \cup \{0\}$ with $i \neq j$, the following holds:
$$
\begin{array}{lcl}
   \iota_i(U_{ji})= Y(i) \cap Y(j)   &  &  \overline{\iota}_i(U_{ji})= Z(i) \cap Z(j). 
\end{array}
$$
Moreover, we have equalities of birational maps 
$$
\begin{array}{lcl}
   \mu_{ij}= \iota_i^{-1} \circ \iota_j, & &   \mu_{ij}= \overline{\iota}_i^{-1} \circ \overline{\iota}_j.
\end{array}
$$
Note that both $Y$ and $Z$ are smooth integral schemes of finite type. 
The natural maps $T(i) \longto  \overline{ T(i)}$ glue to a global open embedding $Y \longto Z$.
From now on, we  identify $Y$ with an open subset of $Z$ trough this map and write $Y \subseteq Z$.
Note that, for any $i \in I_{uf} \cup \{0\}$, we have that $Z(i) \cap Y= Y(i)$. 

\begin{example}
\label{ex: diago comp 1, construction}
 Let $t$ be the seed which is graphically represented by the following quiver 

\[\begin{tikzcd}
	 {\bigcirc 1}  & {\blacksquare 2}
	\arrow[from=1-1, to=1-2]
\end{tikzcd}\]
and whose cluster is $x=(x_1,x_2)$. 

We have that $T(0)= \bigl( \AAA^1 \setminus \{0\} \bigr) \times \AAA^1$, with coordinates $x_1=x(0)_1$ and $x_2=x(0)_2$. Similarly,  $T(1)= \bigl( \AAA^1 \setminus \{0\} \bigr) \times \AAA^1$, with coordinates $x(1)_1$ and $ x(1)_2.$
Moreover, we have that 
$$
U_{10}=  \bigl( \AAA^1 \setminus \{0\} \bigr) \times \bigl( \AAA^1 \setminus \{-1\}) \subseteq T(0) \qquad U_{01}= \bigl( \AAA^1 \setminus \{0\} \bigr) \times \bigl( \AAA^1 \setminus \{-1\}) \subseteq T(1),
$$
and the glueing map $\mu_{10}$ sends a point $(x_1,x_2) $ to $\bigl(\frac{1+x_2}{x_1}, x_2\bigr).$
The morphisms 
$$
\begin{array}{rcl c rcl}
   T(0)  & \longto &  \AAA^2 & &  T(1)  & \longto &  \AAA^2\\[0.7em]
    (x_1, \, x_2) & \longmapsto & \bigl(x_1, \,  \frac{1+x_2}{x_1} \bigr)
    & & 
     (x(1)_1, \, x(1)_2) & \longmapsto & \bigl( \frac{1+x(1)_2}{x(1)_1}, \,  x(1)_1 \bigr)
\end{array}
$$
glue to an open embedding $Y \longto \AAA^2$ whose image is the complement of the origin.

Moreover, we have that $\overline{T(0)}=T(0)$ and $\overline{T(1)}=\AAA^1 \times \AAA^1$. 
Note that the complement of $Y$ in $Z$ is an irreducible closed subspace $E_{1'}$ of $Z$ that, in the chart given by $\overline{T(1)}$, corresponds to the vanishing locus of the function $x(1)_1.$

Observe that, while the finite cluster variety $Y$ is separated,  its diagonal partial compactification $Z$ is not, at least in this example. 
Indeed, recall the natural charts $\overline{\iota}_i : \overline{T(i)} \longto Z \, (1 \leq i \leq 2)$, and consider the morphisms
$$
\begin{array}{rcl c rcl}
    f_0 : \AAA^1 & \longto & \overline{T(0)} & & 
    f_1 : \AAA^1 & \longto & \overline{T(1)} \\
     t & \longmapsto & (1, t) & & t  & \longmapsto & (1+t,t). 
\end{array}
$$
The maps $\overline{\iota}_0 \circ f_0$ and $\overline{\iota}_1 \circ f_1$ agree on the open subset $\AAA^1 \setminus \{-1\}$ of $\AAA^1$, but they are different. 
Indeed, the point $\overline{\iota}_0 \circ f_0 (-1)$ belongs to $Y$, while  $\overline{\iota}_1 \circ f_1 (-1)$ belongs to $E_{1'}$.
It follows that $Z$ is not separated. 
\end{example}

We go back to the general situation.

\begin{lemma}
    \label{lem: fun Y upper}
    For any $i \in I_{uf} \cup \{0\}$, the map  $\iota_i^*$ identifies the ring $\Orb_Y(Y)$ with the upper cluster algebra $\uclu(t)$.
\end{lemma}

\begin{proof}
    Since $t$ is of maximal rank, this is an immediate consequence of Theorem \ref{upper cluster equals upper bound} and the definition of the glueing maps used to construct $Y$.
\end{proof}

The following  lemma is not used in the rest of the paper, but it is conceptually important to relate the previous construction to the more common notion of cluster varieties defined by \cite{fock2009cluster} and developed by
  \cite{gross2013birational}. 
  See for instance Remark \ref{rem: cluster var GHK}.

  \begin{lemma}
  \label{lem:to compare GHK}
      Let $i,j \in I_{uf} \cup \{0\}$ satisfying $i \neq j$.
      The largest open subset of $T(j)$ on which the birational map $\mu_{ij}: T(j) \dashrightarrow T(i)$ is both defined and  an isomorphism with its image is $ U_{ij}.$ 
  \end{lemma}

  \begin{proof}
Recall that, by \cite[Chapitre premier, Sections 6.5, 8.2]{GrothendieckEGA1} if $f:X \dashrightarrow X'$ is a birational map between complex irreducible algebraic varieties, then the largest open subset on which the map $f$ is both defined and an isomorphism with its image coincides with the set of points $x$, of $X$, where $f$ is defined and  is a local isomorphism.
In other words: the set of points $x \in X$ for which the map $f^*:\CC(X')\longto \CC(X)$ restricts to a well defined isomorphism $f_x^*: \Orb_{X',f(x)} \longto \Orb_{X,x}$  for some $f(x) \in X$. Note that in that case, the point $f(x)$ is unique.

 Let $k \in I_{uf}$. It is obvious that the set of points of $T(0)$ (resp. $T(k)$) on which the birational map $\mu_{k0}$ (resp.  $\mu_{0k}$) is defined is $U_{k0}$ (resp. $U_{0k}$).
Since $\mu_{k0}: U_{k0} \longto U_{0k} $ is an isomorphism, the statement follows if one between $i$ and $j$ is zero.

Assume now that $i,j \in I_{uf}$. 
Then, let $x$ be a point of $T(j)$ on which $\mu_{ij}$ is defined. Note that 
$$
\displaystyle \mu_{ij}^*\big(x(i)_j \big)=
\frac{M^+(j)+M^-(j)}{x(j)_j}.
$$
  Therefore, the function $M^+(j)+M^-(j)$ does not vanish on $x$. 
  In other words: we have that 
  $$
  x \in T(j)_{M^+(j)+M^-(j)}=U_{0j}.
  $$
  In particular, the map $\mu_{0j}$ is defined and is a local isomorphism at the point $x$. Moreover, we have that 
  \begin{equation}
  \label{eq:clust var 1}
      \mu_{0j}(x) \in T(0)_{M^+(j)+M^-(j)}=U_{j0}.
  \end{equation}
  Assume furthermore that $\mu_{ij}$ is a local isomorphism at $x$.
  Recalling that $\mu_{ij}=\mu_{i0}\circ \mu_{0j}$, we deduce from the previous discussion that $\mu_{i0}$ is defined and is a local isomorphism at $\mu_{0j}(x)$. 
  Using the part of the statement have proved so far, we deduce that 
  \begin{equation}
      \label{eq:clust var 2}
      \mu_{0j}(x) \in U_{i0}.
  \end{equation}
 Equations \eqref{eq:clust var 1} and \eqref{eq:clust var 2} imply that 
 $$
 x \in \mu_{j0} ( U_{j0} \cap U_{i0})= U_{ij}.
 $$
 Since its clear that $\mu_{ij}$ is an isomorphism on $U_{ij}$, the statement follows.
\end{proof}

\begin{remark}
    \label{rem: cluster var GHK}
     Let $\mathcal{Y}$ be the $\clu$-type cluster variety obtained by applying \cite[Construction 2.10]{gross2013birational} (see also \cite[Proposition 2.4]{gross2013birational}) to the seed $t$. 
     Note that frozen variables are allowed to vanish on $\mathcal{Y}$. 
      Lemma \ref{lem:to compare GHK} directly implies that the scheme $Y$ can be identified with the open subset of $\mathcal{Y}$ obtained by glueing the schemes $T(i)$ for $i \in I_{uf} \cup \{0\}.$
\end{remark}

\subsection{A cluster structure on the Cox ring of Z}

From now on, for any index $i \in I_{uf} \cup \{0\}$, we identify  the spaces $T(i)$ and $Y(i)$ using the map $\iota_i$.
According to this convention, we identify $\Li \bigl( t(i) \bigr)$ with the ring of regular functions on $Y(i)$. 
Moreover, by Lemma \ref{lem: fun Y upper},  we think about the cluster $x(i) $ as a collection of regular functions on $Y$.

\begin{lemma}
    \label{lem:Pic cluster var}
    We have that $\Pic(Y)= \{0\}$, the ring $\Orb_Y(Y)$ is factorial, and $\Orb_Y(Y)^\times= \CC^\times$. Moreover, the following hold.

    \begin{enumerate}
        \item For any $i, j \in I_{uf}$ with $i \neq j$,  the functions $x(0)_i$ and $x(0)_j$ are coprime on $Y.$
        \item  For any $i \in I_{uf}$,  the functions $x(0)_i$ and $x(i)_i$ are coprime on $Y$.
    \end{enumerate}
\end{lemma}

\begin{proof}
    The ring $\Orb_Y(Y)$ is factorial because of Lemma \ref{lem: fun Y upper} and our assumptions.
    Moreover, since $I_{sf}= \emptyset$,  Theorem \ref{thm: inv upper} implies that $\Orb_Y(Y)^\times=\CC^\times$.
   
    Let $j \in I_{uf}$, and $S_j:= Y(j) \setminus Y(0)$. 
    Because of the definition of $Y$, we have that if $i \in   I_{uf} \cup \{0\}$ and $i \neq j$, then $Y(j) \cap Y(i) \subseteq Y(0).$ 
    Therefore, the set $S_j$ is closed in $Y$.
    
 By the definition of $Y$, we have that the zero locus of the function $M^+(j) + M^-(j)$ restricted to the open set $Y(j)$ is $S_j$. 
    Moreover, since the upper cluster algebra $\uclu(t)$ is factorial, the element $M^+(j) + M^-(j)$ is irreducible in the ring  $ \Orb_Y\bigl(Y(j) \bigr) = \Li \bigl( t(j) \bigr)$.
    Indeed, this follows from \cite[Theorem 4.9]{cao2022valuation}.
    Therefore, we have that $S_j$ is irreducible and 
    
    \begin{equation}
        \label{eq:cox 11}
        \divv \bigl( M^+(j) + M^-(j) \bigr)_{| Y(j)}= S_j.
    \end{equation}
    Since $Y$ is smooth, and $Y \setminus Y(0)= \bigcup_{i \in I_{uf}} S_i$, we have an exact sequence 

\[\begin{tikzcd}
	{\Orb_Y\bigl(Y(0) \bigr)^\times} & {\bigoplus_{i \in I_{uf}} \ZZ S_i } & {\Pic(Y)} & {\Pic\bigl(Y(0) \bigr) } & 0.
	\arrow["\divv", from=1-1, to=1-2]
	\arrow[from=1-2, to=1-3]
	\arrow[from=1-3, to=1-4]
	\arrow[from=1-4, to=1-5]
\end{tikzcd}\]
Note that the group $\Pic\bigl(Y(0) \bigr)$ is trivial, as $Y(0)$ is an open subset of an affine space.
Hence, for proving that $\Pic(Y)= \{0\}$, it is sufficient to prove that for any $i \in I_{uf}$, the equality $\divv\bigl( x(0)_i \bigr)= S_i$ holds.

If $j \in I_{uf} \cup \{0\}$ and $j \neq i$, then $\bigl(x(0)_i \bigr)_{| Y(j)}=\bigl(x(j)_i \bigr)_{| Y(j)}$ because of \eqref{eq:glueing clust variety}. 
Therefore, we deduce that 
$$
\divv \bigl (x(0)_i \bigr)_{|Y(j)}=0.
$$
Then, Eq. \eqref{eq:glueing clust variety}, the fact that $x(i)_i \in \Orb_Y\bigl(Y(i)\bigr)^\times$ and Eq. \eqref{eq:cox 11} imply that 
$$
\divv \bigl(x(0)_i \bigr)_{|Y(i)}= \divv \biggl( \frac{M^+(i)+M^-(i)}{x(i)} \biggr)_{|Y(i)}=S_i.
$$
It follows that $\divv\bigl(x(0)_i\bigr)= S_i$ and that $\Pic(Y)= \{0\}.$ 
Regarding the coprimality statements, let $i \in I_{uf}$ and $j \in I_{uf} \setminus \{i\}$. 
The previous discussion shows that $x(0)_i$ and $x(0)_j$ are coprime on $Y$.  Moreover,  the functions $x(0)_i$ and $x(i)_i$ are coprime on $Y$ since $\divv \bigl(x(i)_i \bigr)_{|Y(i)}= 0$. 
\end{proof}


Let $D$ be a second copy of $I_{uf}$. 
We use the notation 
$$
D:=I'_{uf}= \{i': i \in I_{uf}\}.
$$ 
For $i' \in D$, let 
$$
E_{i'}:= Z(i) \setminus Y.
$$
The sets $E_{i'}$ $(i \in I_{uf})$ are closed irreducible subspaces of $Z$ of codimension one.
Indeed, if $i \in I_{uf}$ and $j \in I_{uf} \cup \{0\} \setminus \{i\}$, then
$$
Z(j) \cap Z(i)=Y(j)\cap Y(i) \subseteq Y.
$$ 
Therefore, we have that $E_{i'} \cap Z(j)= \emptyset$, which implies that $E_{i'}$ is closed in $Z$. 
Moreover, in the chart given by $\overline{T(i)}$, the divisor $E_{i'}$ corresponds to the vanishing locus of the function $x(i)_i$. 
Hence,  $E_{i'}$ is irreducible.
Furthermore, we clearly have that $Z \setminus Y= \bigcup_{i' \in D} E_{i'}$.

For $k' \in D$, let $\sigma_{k'}$  be the rational function $1$ seen as global section of the sheaf $\Orb(E_{k'})$. Let $L:=\big(\Orb(E_{k'})\big)_{k' \in D}$ and $\sigma:= (\sigma_{k'})_{k' \in D}$. 

With usual notation, we consider the sheaf $\Rc_L$ on $Z$ and the scheme $\lX_L= \Specbf(\Rc_L)$ associated to the collection of line bundles $L$.
Because of the previous discussion, we have that $(\lX_L,L, \sigma)$ is the standard homogeneously suitable for lifting structure of $\Cox(Z)$ corresponding to the open subset $Y$.

\bigskip

 Let $\lif t^D$ (resp. $\nu$) be the minimal monomial lifting (resp. minimal lifting matrix) of the seed $t$ with respect to  $(\lX_L, L, \sigma)$.
Theorem \ref{thm:min lifting Cox} and Lemma \ref{lem:Pic cluster var} imply that we have an inclusion of $\Pic(Z)$-graded algebras 
$$
\uclu(\lif t^D) \subseteq \Cox(Z).
$$
The rest of the paper is devoted to the proof of the following theorem.

\begin{theo}
    \label{thm:principal clust variety}
   Let $\Id \in \ZZ^{I_{uf} \times I_{uf}}$ be the identity matrix.
   With the natural identification between $D$ and $I_{uf}$,  we have that $\nu= (\Id, 0) \in \ZZ^{D \times (I_{uf} \sqcup I_{hf})}$. Moreover $\uclu(\lif t^D)= \Cox(Z)$.
\end{theo}

For  $i \in I_{uf} \cup \{0\}$, we identify $\overline{T(i)}$ and $Z(i)$ via the map $\overline{\iota_i}$. 
In particular, we think about the cluster $x(i)$ as a collection  of regular functions on $Z(i)$.  

\bigskip

 For $k, i \in I_{uf}$, consider the local parameter $\tau^{(k)}_{i'}$ for the line bundle $\Orb(E_{i'})$, on the open subset $Z(k)$, defined by 
 
 \begin{equation}
    \tau^{(k)}_{i'}:= 
    \begin{cases}
        1 & \text{if} \quad i\neq k\\
        x(k)_k^{-1} & \text{if} \quad i= k. 
    \end{cases}
\end{equation}
We have that

\begin{equation}
\label{eq 7 Cox}
    \frac{\sigma_{i'}}{\tau^{(k)}_{i'}}= 
    \begin{cases}
        1 & \text{if} \quad i\neq k\\
        x(k)_k & \text{if} \quad i= k. 
    \end{cases}
\end{equation}

As explained in Section \ref{sec:charact space} (see Formula \eqref{eq:local triv ssheaf of line b}), the collection of local parameters $\tau^{(k)}:=(\tau^{(k)}_{i'})_{i' \in D}$ gives a  $T$-equivariant open embedding $\phi^{(k)} : T \times Z(k) \longto \lX_L$. 
Similarly, recall that the open embedding determined by the collection of sections $\sigma$ (Eq. \eqref{eq: phi standard suitable}) is denoted by $\phi : T \times Y \longto \lX_L$.
Consider the commutative diagram

\[\begin{tikzcd}
	& {\lX_L} \\
	{T \times Y} && {T \times Z(k).}
	\arrow["\phi", from=2-1, to=1-2]
	\arrow["{\phi^{(k)}}"', from=2-3, to=1-2]
	\arrow[dashed, from=2-3, to=2-1]
\end{tikzcd}\]
Observe that the pullback of rational functions along the birational map $T \times Z(k) \dashrightarrow T \times Y$
corresponds to the field morphism induced by the map $\CC[T] \otimes \CC(Z) \longto \CC[T] \otimes \CC(Z) $ which is defined, on homogeneous elements, by
\begin{equation}
    \label{eq: fake change charts Cox}
    \varpi_n \otimes f \longmapsto \varpi_n \otimes  f \biggl( \frac{\sigma}{\tau^{(k)}}\biggr)^n \qquad (n \in \ZZ^D). 
\end{equation}
  Recall that, according to our general notation, we have that 
  $$
  \biggl(\frac{\sigma}{\tau^{(k)}}\biggr)^n= \prod_{i' \in D} \biggl(\frac{\sigma_{i'}}{\tau^{(k)}_{i'}}\biggr)^{n_{i'}} \qquad (n \in \ZZ^D).
  $$
  Using \eqref{eq 7 Cox}, we compute that Eq. \eqref{eq: fake change charts Cox} in fact given by

\begin{equation}
    \label{eq:change charts Cox}
    \varpi_n \otimes f \longmapsto \varpi_n \otimes f  \bigl(x(k)_k\bigr)^{n_{k'}} \qquad (n \in \ZZ^D). 
\end{equation}

\begin{remark}
    Recall that $t(0)=t$. 
    Therefore, to avoid awkward notation, from now on we write $x=(x_i)_{i \in I}$ (resp. $\lif x= (\lif x_i)_{i \in \lif \spc I}$) for the cluster variables of the seed $t$ (resp. $\lif t^D)$.
    Recall that $\lif I=I \cup D.$ 
\end{remark}

\begin{lemma}
    \label{lem:computation nu Cox cluster var}
    We have that $\nu= (\Id, 0) \in \ZZ^{D \times (I_{uf} \sqcup I_{hf})}$. Moreover, for any $k \in I_{uf}, \, i \in I$, and $j' \in D$, we have that

    $$
    \biggl(\phi^{(k)}\biggr)^*(\lif x_i)= \begin{cases}
    1 \otimes x(k)_i & \text{if} \quad i \in I_{hf}\\
        \varpi_{e_{i'}} \otimes x(k)_i & \text{if} \quad i \in I_{uf} \setminus \{k\} \\
          \varpi_{e_{k'}} \otimes\big(M^+(k) + M^-(k)\big) & \text{if} \quad i =k,
    \end{cases} 
    $$
    and 
    $$
    \biggl(\phi^{(k)}\biggr)^*(\lif x_{j'})= \begin{cases}
        \varpi_{e_{j'}} \otimes 1 & \text{if} \quad j \neq k\\
          \varpi_{e_{k'}} \otimes x(k)_k & \text{if} \quad j =k.
    \end{cases}
    $$
\end{lemma}

\begin{proof}
  Recall that $p: \lX_L \longto Z$ denotes the natural projection, and  $\val_{k'} : \CC(\lX_L) \longto \ZZ \cup \{\infty\}$ is the divisorial valuation associated to the divisor $p^*(E_{k'})$.
   Note that the divisor $p^*(E_{k'})$, along the chart given by $\phi^{(k)}$, corresponds to the closed subset $T \times \{ x(k)_k = 0\}$ of $T \times Z(k).$
    Recall that we identify $T \times Y$ with its image in $\lX_L$ by means of the map $\phi$, and this convention allows to think at $1 \otimes x_i $ as a rational function on $\lX_L$.

    We compute $\biggl(\phi^{(k)} \biggr)^*( 1 \otimes x_i)$ using \eqref{eq:change charts Cox}. 
    Since $1 \otimes x_i$ is $X(T)$-homogeneous of degree zero, we deduce, using the glueing data \eqref{eq:glueing clust variety}, that  
    $$ 
    \biggl(\phi^{(k)} \biggr)^*( 1 \otimes x_i)= \begin{cases}
        1 \otimes x(k)_i & \text{if} \quad i \neq k\\[0.6em]
        1 \otimes \frac{M^+(k)+ M^-(k)}{x(k)_k} & \text{if} \quad i=k.
    \end{cases}
    $$
    It follows at once that $\val_{k'}( 1 \otimes x_i)= - \delta_{ki}$, from which the description of the matrix $\nu$ follows.

    \bigskip

    Recall that $\lif x_{j'}= \sigma_{j'}$ and that $\lif x_{j'}$ is $X(T)$-homogeneous of degree $\varpi_{e_{j'}}.$ 
    Then, the desired expression for $\biggl(\phi^{(k)} \biggr)^*(\lif x_{j'})$ follows at once from \eqref{eq:change charts Cox}. 
    Moreover, because of the formula for $\nu$ that we have just proved, we have that 
    $$
    \lif x_i = \begin{cases}
        (1 \otimes x_i) & \text{if} \quad i \not \in I_{uf}\\
        \lif x_{i'}(1 \otimes x_i) & \text{if} \quad i \in I_{uf}.
    \end{cases}
    $$
   The desired expression for $\biggl(\phi^{(k)}\biggr)^*(\lif x_i)$ is then an immediate consequence of the previous discussion.
\end{proof}

\begin{coro}
    \label{cor: lif B partial com.}
    The generalized exchange matrix $\lif B$ of the seed $\lif t^D$ is 
$$
\lif B= \bmat B\\
-B^\circ 
\emat,
$$
where $B^\circ$ is the principal part of $B$.
\end{coro}
\begin{proof}
    By Lemma \ref{lem:computation nu Cox cluster var}, we have that $-\nu B= -B^\circ$.
\end{proof}

\begin{example}
    Let $t$ be the seed which is graphically represented by the following quiver 

\[\begin{tikzcd}
	{\bigcirc 1} & {\bigcirc 2} & {\bigcirc 3,}
	\arrow[from=1-2, to=1-3]
	\arrow[from=1-1, to=1-2]
\end{tikzcd}\]
and whose cluster is $x=(x_1,x_2,x_3)$. 
Then, according to the previous notation, the seed $\lif t^D$  is graphically represented by the quiver
    \[\begin{tikzcd}
	{\blacksquare 1'} & {\blacksquare 2'} & {\blacksquare 3'} \\
	{\bigcirc 1} & {\bigcirc 2} & {\bigcirc 3}
	\arrow[from=2-2, to=2-3]
	\arrow[from=2-1, to=2-2]
	\arrow[from=2-2, to=1-1]
	\arrow[from=2-3, to=1-2]
	\arrow[from=1-2, to=2-1]
	\arrow[from=1-3, to=2-2]
\end{tikzcd}\]
and its cluster variables are

$$ \begin{array}{ccc}
   \lif x_{1'}= \sigma_{1'} & \lif x_{2'}= \sigma_{2'} & \lif x_{3'}= \sigma_{3'}\\
   \lif x_1=x_1 \sigma_{1'} & \lif x_2= x_2 \sigma_{2'} & \lif x_3= x_3 \sigma_{3'}.
\end{array}$$
The degree of the cluster variables, with respect to the canonical graduation of the seed $\lif t^D$ is given, with obvious notation, below.
$$ \begin{array}{ccc}
  \varpi_{e_{1'}} & \varpi_{e_{2'}} & \varpi_{e_{3'}}\\
    \varpi_{e_{1'}} & \varpi_{e_{2'}} & \varpi_{e_{3'}}.
\end{array}$$
\end{example}

\begin{proof}[Proof of Theorem \ref{thm:principal clust variety}]
The statement on $\nu$ has been proved in Lemma \ref{lem:computation nu Cox cluster var}. Suppose by contradiction that $\uclu(\lif t^D)$ is strictly contained in $\Cox(Z)$. 
By Proposition \ref{prop: equality conditions}, there exists an index $k' \in D$ and an element $f \in \Cox(Z)$ such that $\cval_{k'}(f) < 0$. 
The  proof consists of 3 steps.

 \bigskip

\noindent \underline{Step 1:} Prove that $f$ can be assumed as in \eqref{eq 3cox}.

\noindent  \underline{Step 2:} Prove that we can assume  \eqref{eq:cox n_d=0},\eqref{eq 4cox} \eqref{eq:cox Z 3}.

\noindent  \underline{Step 3:} Prove that any coefficient $a_n$ of \eqref{eq 3cox} is zero, which contradicts the fact that $f $ belongs to $ \Cox(Z) \setminus \uclu(\lif t^D).$

\bigskip

\underline{Step 1.} Since the ring $\Cox(Z)$ is $X(T)$-graded, we can suppose that $f $ is homogeneous.
By Theorem \ref{thm: min mon lifting}, the ring $\Cox(Z)$ is contained in the upper cluster algebra $\uclu(\lif t)$.
Hence, recalling Eq. \eqref{eq: localising clust min lifting}, we deduce that up to multiplying by a monomial in the variables  $\lif x_{j'} \, (j'\in D)$ and in the unfrozen variables of the cluster $\lif x$,  we can suppose that 

    \begin{equation}
    \label{eq 3cox}
        f= \frac{P}{\lif x_{k'}}, \quad \text{where:} \qquad P= \sum_{n  \in \NN^{\lif \spc I}}a_n \lif x^n  \quad \text{and} \quad \lif x_{k'} \not | \, P.
    \end{equation}
Hence, $P$ is a polynomial in the variables of $\lif x$ which is not divisible by $\lif x_{k'}$.
Divisibility is intended in the polynomial ring in the cluster variables.
Moreover, since $f$ is homogeneous, we have  that  $f \in \Cox(Z)_{\varpi_{N}}$ for some $N \in \ZZ^D$. 

\underline{Step 2.} Because of Theorem \ref{laurent pheno}, up to adding to $f$ a polynomial in the variables of the cluster $\lif x$, we can assume that the following condition holds:

\begin{equation}
\label{eq:cox n_d=0}
    \text{if} \, \, n \in \NN^{\lif \spc I} \, \, \text{and} \,\, a_n \neq 0 ,  \, \, \text{then} \, \,n_{k'} = 0.
\end{equation}
In other words, we assume that the sum defining $P$ runs over the set $\NN^{\lif \spc I \setminus \{k' \}}$, which is canonically identified with the set of elements of $\NN^{\lif \spc I}$ whose coordinate indexed by $k'$ equals zero.

For $ n \in \ZZ^{\lif \spc I \setminus \{k'\}}$, we denote by $n^I \in \ZZ^{I}$ and $n' \in \ZZ^{D \setminus \{k'\}}$ the two coordinates of $n$ with respect to the decomposition $\lif I \setminus \{k'\}= I \sqcup (D \setminus \{k'\}).$
 The fact that $f$ is homogeneous of degree $\varpi_N$ imposes, using the expression for $\nu$ given by Lemma \ref{lem:computation nu Cox cluster var}), that

\begin{equation}
    \label{eq 4cox}
   \text{if} \, \, n \in \NN^{\lif \spc I \setminus \{k' \}} \, \, \text{and} \, \, a_n \neq 0, \, \, \text{then} \quad  \varpi_{\nu n^I + n'}= \varpi_N + \varpi_{e_{k'}}. \quad \text{That is:} \quad \nu n^I + n'=N + e_{k'}.
\end{equation}
We stress that, in Formula \eqref{eq 4cox}, $\ZZ^{D \setminus \{k'\}}$ is naturally identified with the subset of $\ZZ^{D}$ of vectors whose coordinate indexed by $k'$ is zero.

We denote by $x(k)_{\setminus k} \in \CC(Z)^{I \setminus \{k\}}$ the collection of elements obtained by forgetting the variable $x(k)_k$ form the cluster $x(k)$.
Moreover, for $n \in \ZZ^{\lif \spc I \setminus \{k\}}$, we denote by $n^{I}_{\setminus k} \in \ZZ^{I \setminus \{k\}}$ the vector obtained by forgetting the $k$-th coordinate of $n^I$.
According to our conventions, the notation $x(k)_{\setminus k}^{n^I_{\setminus k}}$ stands for a well defined Laurent monomial in the cluster variables of $x(k)$ different from $x(k)_k$.

 Lemma \ref{lem:computation nu Cox cluster var}, Eq. \eqref{eq 3cox}, \eqref{eq:cox n_d=0}, and \eqref{eq 4cox} imply that 
\begin{equation}
    \label{eq:6 Cox}
 \biggl(\phi^{(k)}\biggr)^*(P)= \varpi_{N+ e_{k'}} \otimes \sum_{n \in \NN^{\lif \spc I \setminus \{k'\}}} a_n  x(k)_{\setminus k}^{n^I_{\setminus k}} \bigl( M^+(k) + M^-(k)\bigr)^{n^I_k}.
 \end{equation}
As the function $f$ is regular on $\lX_L$ and $ \biggl(\phi^{(k)}\biggr)^*(\lif x_{k'})= \varpi_{e_{k'}} \otimes x(k)_k$, we deduce that the right hand side of \eqref{eq:6 Cox} is divisible, in the ring $\Orb_{T \times {Z(k)}}( T \times {Z(k)})$, by the element $\varpi_{e_{k'}} \otimes x(k)_k$.
Considering the $e_{k'}$-coefficient in \eqref{eq 4cox}, we deduce that:

\begin{equation}
    \label{eq:cox Z eq 2}
     \text{if} \, \, n \in \NN^{\lif \spc I \setminus \{k' \}} \, \, \text{and} \, \, a_n \neq 0, \quad \text{then} \quad n^I_{k}= N_{k'}+1.
\end{equation} 

Therefore, in the right hand side of \eqref{eq:6 Cox}, we can factor out the term 
$
\bigl( M^+(k) + M^-(k)\bigr)^{n^I_k}.
$
Moreover, since $M^+(k) + M^-(k)$ is coprime with $x(k)_k$, the previously mentioned divisibility is equivalent the fact that the element

$$
\sum_{l \in \ZZ^{I \setminus \{k\}}}  \biggl( \sum_{ \substack{n \in \NN^{\lif \spc I \setminus \{k'\}} \\
n^I_{\setminus k}=l}} a_n \biggr) x(k)_{\setminus k}^l
$$
is divisible by $x(k)_k$. 
In other words, we can simplify the term 
$
\bigl( M^+(k) + M^-(k)\bigr)^{n^I_k}.
$
Here, divisibility is intended in the polynomial ring $\CC[x(k)]$ in the cluster variables of $x(k).$
By the definition of $x(k)_{\setminus{k}},$ any monomial of the form $x(k)_{\setminus{k}}^l \, (l \in \ZZ^{I \setminus \{k\}}), $ does not have $x(k)$ in its support.
It follows that:
\begin{equation}
    \label{eq:cox Z 3}
\text{for any} \, \, l \in \ZZ^{I \setminus \{k\}}, \, \, \text{then} \quad  \sum_{ \substack{n \in \NN^{\lif \spc I \setminus \{k'\}} \\
n^I_{\setminus k}=l}} a_n=0.
\end{equation}

\underline{Step 3.} For any $l  \in \ZZ^{I \setminus \{k\}} $, it is immediate to verify, using  \eqref{eq 4cox}, that there is at most one index $n \in \NN^{\lif \spc I \setminus{k'}}$ satisfying $a_n \neq 0$ and $n^I_{\setminus k}= l$. 
This implies that, for any $n  \in \NN^{\lif \spc I \setminus \{k'\}}$, then $a_n=0$.
Using \eqref{eq:cox n_d=0}, it follows that $f=0$, which is a contradiction.
\end{proof}

\begin{example}
    We continue Example \ref{ex: diago comp 1, construction}. 
   We have that $D= \{1'\}$, and the seed $\lif t^D$ is graphically represented by the quiver
   \[\begin{tikzcd}
	{\blacksquare 1'} & {} \\
	{\bigcirc 1 } && {\blacksquare 2,}
	\arrow[from=2-1, to=2-3]
\end{tikzcd}\]
with cluster variables 
$$
\lif x_{1'}= \sigma_{1'}, \quad \lif x_1= \sigma_{1'} x_1, \quad \lif x_2 = x_2.
$$
The degree of the initial cluster variables, as homogeneous elements of $\Cox(Z)$, is respectively
$$
\varpi_{e_{1'}}, \quad \varpi_{e_{1'}}, \quad 0.
$$
Let $\lif x(1)_1$ be the cluster variable associated to the vertex $1$ in the seed $\mu_1(\lif t^D).$
The element $\lif x(1)_1$ is homogeneous and its degree is $-\varpi_{e_{1'}}$.
Observe that the cluster algebra $\clu(\lif t^D)$ equals the upper cluster algebra $\uclu(\lif t^D)$ by \cite[Proposition 3.3]{bucher2019upper}.
Therefore, since the vertex $1$ is the only mutable vertex of the seed $\lif t^D$, Theorem \ref{thm:principal clust variety} implies that $\Cox(Z)$ is generated, as a $\CC$-algebra, by the homogeneous elements
$$
\sigma_{1'}, \quad \lif x_1, \quad \lif x_2, \quad \lif x(1)_1,
$$
subject to the relation 
$$
\lif x_1 \lif x(1)_1= 1 + \lif x_2.
$$
In particular, $\Cox(Z)$ is finitely generated as a $\CC$-algebra. 
Recall that, in Example \ref{ex: diago comp 1, construction}, we showed that the scheme $Z$ is not separated. 
Therefore, $Z$ is not quasi-projective.
Nevertheless, by \cite[Theorem 1]{hausen2002equivariant} and the finite generation of $\Cox(Z)$, the scheme $Z$ admits a closed embedding into a smooth toric prevariety of affine intersection. 
\end{example}

\bibliographystyle{alpha}

\bibliography{biblio}

\end{document}